\documentclass[12pt,reqno,twoside]{amsart}
\usepackage{amssymb,amsfonts,amsthm,amsmath,mathrsfs}
\usepackage{cite}
\usepackage[left=1 in,top=1 in,right=1 in,bottom=1 in]{geometry}
\usepackage{enumitem}
\usepackage{color}
\usepackage{mathtools}
\mathtoolsset{showonlyrefs}

\usepackage{hyperref}

\newcommand{\beq}{\begin{equation}}
\newcommand{\eeq}{\end{equation}}
\newcommand{\beqs}{\begin{equation*}}
\newcommand{\eeqs}{\end{equation*}}
\newcommand{\ba}{\begin{array}}
\newcommand{\ea}{\end{array}}
\newcommand{\beas}{\begin{eqnarray*}}
\newcommand{\eeas}{\end{eqnarray*}}
\newcommand{\bea}{\begin{eqnarray}}
\newcommand{\eea}{\end{eqnarray}}
\newcommand{\bal}{\begin{align}}
\newcommand{\eal}{\end{align}}

\newcommand{\bals}{\begin{align*}}
\newcommand{\eals}{\end{align*}}

\newcommand{\K}{\ensuremath{\mathbb K}}
\newcommand{\R}{\ensuremath{\mathbb R}}

\newcommand{\C}{\ensuremath{\mathbb C}}
\newcommand{\N}{\ensuremath{\mathbb N}}
\newcommand{\Z}{\ensuremath{\mathbb Z}}

\newcommand{\bigo}{\mathcal O}
\newcommand{\LL}{\mathcal L}
\newcommand{\iln}{L}

\newcommand{\inprod}[1]{\langle{#1}\rangle}
\newcommand{\doubleinprod}[1]{\langle\!\langle{#1}\rangle\!\rangle}

\newcommand{\bds}{\begin{displaystyle}}
\newcommand{\eds}{\end{displaystyle}}

\newcommand{\spec}{\mathfrak S}

\def\eqdef{\stackrel{\rm def}{=}}

\def\d{{\rm d}}

\newcommand{\bvec}[1]{\mathbf{#1}}
\def\vecx{\bvec x}

\def\vece{\bvec e}
\def\vecu{\bvec u}
\def\vecv{\bvec v}
\def\vecx{\bvec x}
\def\vecw{\bvec w}

\def\veck{\bvec k}

\def\varep{\varepsilon}
\renewcommand{\Re}{\operatorname{Re}}
\renewcommand{\Im}{\operatorname{Im}}
\def\ddt{\frac{\d}{\d t}}

\newcommand{\vkL}{\veck_{\mathbf L}}

\newtheorem{theorem}{Theorem}[section]

\newtheorem{lemma}[theorem]{Lemma}

\newtheorem{proposition}[theorem]{Proposition}
\newtheorem{definition}[theorem]{Definition}
\newtheorem{assumption}[theorem]{Assumption}

\theoremstyle{definition}
\newtheorem{example}[theorem]{Example}
\newtheorem{scenario}[theorem]{Scenario}
\newtheorem{remark}[theorem]{Remark}

\newcommand{\tnum}{\rm(\roman*)}
\newcommand{\rnum}{\rm(\alph*)}

\date{\today}
\numberwithin{equation}{section}

\title{Asymptotic expansions with subordinate variables for solutions of the Navier--Stokes equations}

\author{Luan Hoang}

\address{Department of Mathematics and Statistics,
Texas Tech University\\
1108 Memorial Circle, Lubbock, TX 79409--1042, U. S. A.}
\email{luan.hoang@ttu.edu}

\makeatletter
\@namedef{subjclassname@2020}{
  \textup{2020} Mathematics Subject Classification}
\makeatother

\begin{document}

\begin{abstract} 
We study the three-dimensional Navier--Stokes equations in a periodic domain with the force decaying in time. Although the force has a certain coherent decay, as time tends to infinity, it can be too complicated for the previous theory of asymptotic expansions to be applicable. To deal with this issue, we systematically develop a new theory of asymptotic expansions containing the so-called subordinate variables which can be defined recursively. We apply it to obtain an asymptotic expansion for any Leray--Hopf weak solutions. The expansion, in fact, is constructed explicitly and the impact of the subordinate variables can be clearly specified. The complexifications of the Gevrey--Sobolev spaces, and of the Stokes and bilinear operators of the  Navier--Stokes equations are utilized to facilitate such a construction.
\end{abstract}

\maketitle


\tableofcontents
 
\pagestyle{myheadings}\markboth{L. Hoang}{Asymptotic expansions with subordinate variables for the Navier--Stokes equations}

\section{Introduction}\label{intro}

We study the long-time dynamics of viscous, incompressible fluids in the three-dimensional space.
We denote by 
$\vecx\in \R^3$ the spatial variables,  
$t\in\R$ the time variable, 
$\nu>0$,  the (kinematic) viscosity,
$\vecu(\vecx,t)\in\R^3$ the  velocity vector field, 
$p(\vecx,t)\in\R$ the pressure, 
and  $\mathbf f(\vecx,t)\in\R^3$ the body force. 
The Navier--Stokes equations (NSE) describing the fluid flows are 
\begin{align}\label{nse}
\begin{split}
&\bds \frac{\partial \vecu}{\partial t}\eds  + (\vecu\cdot\nabla)\vecu -\nu\Delta \vecu = -\nabla p+\mathbf f \quad\text{on }\R^3\times(0,\infty),\\
&\textrm{div } \vecu = 0  \quad\text{on }\R^3\times(0,\infty),
\end{split}
\end{align}
with the initial condition
\beq\label{ini}
\vecu(\vecx,0) = \vecu^0(\vecx),
\eeq 
where  $\vecu^0(\vecx)$ is a given divergence-free vector field. 
Throughout the paper, we use the  notation  
$u(t)=\vecu(\cdot,t)$, $f(t)=\mathbf f(\cdot,t)$ and $u^0=\vecu^0(\cdot).$ 

When the force is potential, i.e., $\mathbf f(\vecx,t)=-\nabla \phi(\vecx,t)$ for some scalar function $\phi$, 
Foias and Saut prove in  \cite{FS87}  that  any  non-trivial, regular solution
$u(t)$, in a bounded or periodic domain $\Omega$, admits an  asymptotic expansion (as $t\to\infty$)
\beq \label{expand}
u(t) \sim \sum_{n=1}^\infty q_n(t)e^{-\mu_nt}
\eeq
in Sobolev spaces $H^m(\Omega)^3$, for all $m\ge 0$. Here, the functions $q_n(t)$ are $C^\infty(\Omega)^3$-valued polynomials in $t$,
and $\mu_n>0$ is strictly increasing to infinity.
The reader is referred to \cite{FHS2} for a survey on the subject, or more specifically,
\cite{FS84a,FS84b,FS87,FS91} for the normalization map and invariant nonlinear manifolds associated with \eqref{expand}, 
\cite{FHOZ1,FHOZ2,FHS1} for the corresponding Poincar\'e--Dulac normal form,
\cite{FHN1,FHN2} for applications to the analysis of helicity, statistical solutions, and decaying turbulence.
The expansion \eqref{expand} is improved to hold in the Gevrey classes \cite{HM1}, and extended to include the decaying non-potential forces \cite{HM2,CaH1,CaH2,H6} with increasing level of complexity.
The newer expansions can be complicated containing the power or logarithmic or iterated logarithmic functions of time. Asymptotic expansions similar to \eqref{expand} are also established for the Lagrangian trajectories associated with the solutions of the NSE  \cite{H4}.
Furthermore, they are obtained for solutions of dissipative wave equations \cite{Shi2000}. 
This theory has been developed  for  general nonlinear ordinary differential equations (ODEs), specifically,   for  autonomous analytic systems \cite{Minea},  systems with non-smooth nonlinearity  \cite{CaHK1}, and for nonautonomous systems \cite{CaH3,H5,H10}.
 
Among the papers last mentioned, \cite{H10} is our particular interest.
In that article,  the asymptotic expansion, as $t\to\infty$,  at first cannot be obtained by the previous techniques in \cite{H5,H6,CaH3}. The reason behind this failure is that each of its possible terms may not be in a closed form.
Then a subordinate variable $\zeta\in\R$ is introduced to make it closed. With this variable $\zeta$, all decaying solutions admit a  new type of asymptotic expansion which is still meaningful and provides precise asymptotic approximations for the solutions.
The current work aims to  extend and develop that method systematically for the NSE.
We  explain the main ideas here.
The asymptotic expansion, as $t\to\infty$,  for the force $f(t)$ in the NSE and, later, any Leray--Hopf weak solution $u(t)$, will be of the form
\beq\label{fpol}
\sum_{n=1}^\infty \phi_n(t)\psi(t)^{-\mu_n}\text{ in some Gevrey--Sobolev space $G_{\alpha,\sigma}$,}
\eeq
where $\psi(t)=t$, or $\ln t$, or some iterated logarithmic function of $t$. 
Above, the main decaying modes are $\psi(t)^{-\mu_n}$, with positive numbers $\mu_n\nearrow \infty$.
Each function $\phi_n(t)$ is  $G_{\alpha,\sigma}$-valued and of the form 
\beq\label{phi}
\phi_n(t)=g_n\left (e^t,t,\ln t,\ln\ln t,\ldots,\mathcal Y_1(t),\mathcal Y_2(t),\mathcal Y_3(t),\ldots\right),
\eeq 
where
$g_n=g_n(z,\zeta)$
with $z$ being multi-variables  and  $\zeta=(\zeta_1,\zeta_2,\zeta_3,\ldots)$.
The variable $z$ is substituted by $(e^t,t,\ln t,\ln\ln t,\ldots)$ in \eqref{phi}.
The variables $\zeta_k$ are called subordinate variables and  have the following recursive substitutions 
$$\zeta_1=Z_1(z),
\quad  \zeta_2=Z_2(z,\zeta_1),
\quad \zeta_3=Z_3(z,\zeta_1,\zeta_2),
\quad \ldots.$$
The   substitutions in $t$ for $\zeta$ yield the functions $\mathcal Y_k(t)$ in \eqref{phi}, which are
\begin{align*}
\mathcal Y_1(t)&=Z_1(e^t,t,\ln t,\ln\ln t,\ldots),\quad
\mathcal Y_2(t)=Z_2(e^t,t,\ln t,\ln\ln t,\ldots,\mathcal Y_1(t)),\\
\mathcal Y_3(t)&=Z_3(e^t,t,\ln t,\ln\ln t,\ldots,\mathcal Y_1(t),\mathcal Y_2(t)),
\quad \ldots.
\end{align*}
Recall that  $\zeta$ in \cite{H10} is a single variable in $\R$ and the functions $g_n$ and $Z_k$ only have real number powers for $\zeta$. Here, $\zeta$ represents multi-variables and can have complex number powers in  $g_n$ and $Z_k$. Therefore, \eqref{fpol} and \eqref{phi} cover a much wider and more complicated  class of asymptotic expansions than those studied  in \cite{H10} and also \cite{CaH1,CaH3,H5,H6}.

The paper is organized as follows.
In Section \ref{Prelim} we recall the standard functional setting for the NSE including the Stokes operator $A$, the bilinear mapping $B$ and the Gevrey--Sobolev spaces $G_{\alpha,\sigma}$. We also review their correpsonding complexifications $A_\C$, $B_\C$ and $G_{\alpha,\sigma,\C}$ which are crucial in later statements and proofs.
In Section \ref{classtype}, we introduce the subordinate variables and, in Definition \ref{Hclass}, the main class of functions containing them.
A class of substitution functions for the subordinate variables are defined in Definition \ref{plusclass}.
Systems of subordinate variables and their substitutions are specified in Definition \ref{Zsys}. 
It is worthy to note that these variables, or more precisely, their substitution functions can be defined recursively.
In Section \ref{results}, we define  important linear operators $\mathcal M_j,\mathcal R,\mathcal Z_{A_\C}$  in Definitions \ref{defMR} and \ref{defZA}.
They are used, under the main Assumption \ref{B1},  to explicitly  construct in Definition \ref{construct} the   formal terms of the asymptotic expansion for  any Leray--Hopf weak solution $u(t)$.
The subordinate variables have a recursive effect which is expressed explicitly in the products
$W_j \partial q_\lambda/\partial \zeta_j$ in formula \eqref{chin} when $m_*=0$.
The main results are Theorems \ref{mainthm} and \ref{thm3} that establish the  asymptotic expansion for $u(t)$.
The former theorem is stated using complexified spaces.
The latter one only involves the standard real Gevrey--Sobolev spaces, but  the classes of functions and definition of the asymptotic expansions need be reformulated in Definitions \ref{realPL}--\ref{realex}.
In subsection \ref{example}, we illustrate the obtained results with some typical situations for the force $f(t)$ and some nontrivial examples.
Section \ref{prepsec} contains necessary preparations for the proofs of the main results.
Elementary  properties of the involved functions and linear operators are obtained in subsection \ref{mainprops}.
More importantly, we establish in Theorem \ref{Fode} an asymptotic approximation for solutions of  the linearized NSE. It is a key technical step to recursively obtain each term of the asymptotic expansion.
Finally the proofs of the main results are given in Section \ref{proofsec}. For Theorem \ref{mainthm}, 
 essential properties in Proposition \ref{qregpower} for the terms of the solutions' asymptotic expansion  are needed and hence proved. 
Theorem \ref{thm3} is proved by using Theorem \ref{mainthm}  and the conversions in Lemma \ref{convert} between corresponding functions, subordinate variables and asymptotic expansions.

\section{Backgrounds}\label{Prelim}

We  use the following notation throughout the paper. 
\begin{itemize}
 \item $\N=\{1,2,3,\ldots\}$ denotes the set of natural numbers, $\Z_+=\N\cup\{0\}$, and number $i=\sqrt{-1}$.

 \item For any vector $z\in\C^n$, the real part, respectively, imaginary part, conjugate vector, Euclidean norm, of $z$ is denoted by $\Re z$, respectively, $\Im z$, $\bar z$, $|z|$.
 
 \item Denote 
 $ \mathbf 0_n=(0,0,\ldots,0)\in \R^n$ and $\mathbf 1_n=(1,1,\ldots,1)\in \R^n.$
   
 \item Let $f$ and $h$ be non-negative functions on $[T_0,\infty)$ for some $T_0\in \R$. We write 
 $$f(t)=\bigo(h(t)), \text{ implicitly meaning as $t\to\infty$,} $$
 if there exist numbers $T\ge T_0$ and $C>0$ such that $f(t)\le Ch(t)$ for all $t\ge T$.
\end{itemize}

Let $S$ be a subset of $\C^n$. 
\begin{itemize}
\item We say $S$ preserves the addition if $x+y\in S$ for all $x,y\in S$.
\item We say $S$ preserves the conjugation if  $\overline z\in S$ for any $z\in S$.
\item When $n=1$, we say $S$ preserves the unit increment if $z+1\in S$ for all $z\in S$.
\end{itemize}

\subsection{The functional setting}\label{setting}

We recall the standard functional setting for the NSE in periodic domains, see e.g. \cite{CFbook, TemamAMSbook,TemamSIAMbook, FMRTbook}.
Let $\ell_1$, $\ell_2$, $\ell_3$ be fixed positive numbers. Denote $\mathbf L= (\ell_1,\ell_2,\ell_3)$, $\ell_*=\max\{\ell_1,\ell_2,\ell_3\}$, and the domain $\Omega=(0,\ell_1)\times(0,\ell_2)\times(0,\ell_3)$.
Let  $\{\vece_1,\vece_2,\vece_3\}$ be the standard basis of $\R^3$. A function $g(\vecx)$ is said to be $\Omega$-periodic if
\beqs
g(\vecx+\ell_j \vece_j)=g(\vecx)\quad \textrm{for all}\quad \vecx\in \R^3,\ j=1,2,3,\eeqs
 and is said to have zero average over $\Omega$ if 
\beqs
\int_\Omega g(\vecx)d\vecx=0.
\eeqs

In this paper, we  focus on the case when the force $\mathbf f(\vecx,t)$ and solutions $(\vecu(\vecx, t),p(\vecx, t))$  are $\Omega$-periodic.
By rescaling the variables $\vecx$ and $t$, we assume throughout, without loss of generality, that  $\ell_*=2\pi$ and $\nu =1$.

Let $L^2(\Omega)$ and $H^m(\Omega)=W^{m,2}(\Omega)$, for integers $m\ge 0$, denote the standard Lebesgue and Sobolev spaces on $\Omega$.
The standard inner product and norm in $L^2(\Omega)^3$ are denoted by $\inprod{\cdot,\cdot}$ and $|\cdot|$, respectively.
(We warn that this  notation  $|\cdot|$ also denotes the Euclidean norm in $\R^n$ and $\C^n$, for any $n\in\N$, but its meaning will be clear based on the context.)

Let $\mathcal{V}$ be the set of all $\Omega$-periodic trigonometric polynomial vector fields which are divergence-free and  have zero average over $\Omega$.  
Define
$$H, \text{ respectively } V,\ =\text{ closure of }\mathcal{V} \text{ in }
L^2(\Omega)^3, \text{ respectively } H^1(\Omega)^3.$$
Thanks to the zero average condition, it is standard to  use the following inner product on $V$
\beq\label{Vprod}
\doubleinprod{\vecu,\vecv}
 =\sum_{j=1}^3 \inprod{ \frac{\partial \vecu}{\partial x_j} , \frac{\partial \vecv}{\partial x_j} }\quad \text{for all } \vecu,\vecv\in V,
\eeq
and denote its corresponding norm by $\|\cdot\|$.
We use the following embeddings and identification
$$V\subset H=H'\subset V',$$
where each space is dense in the next one, and the embeddings are compact.

Let $\mathcal{P}$ denote the orthogonal (Leray) projection in $L^2(\Omega)^3$ onto $H$.
The Stokes operator $A$ is a bounded linear mapping from $V$ to its dual space $V'$ defined by  
$$\inprod{A\vecu,\vecv}_{V',V}=\doubleinprod{\vecu,\vecv}\text{ for all }\vecu,\vecv\in V.$$
As an unbounded operator on $H$, the operator $A$ has the domain $\mathcal D(A)=V\cap H^2(\Omega)^3$, and, under  the current consideration of periodicity condition, 
\beq\label{ADA} A\vecu = - \mathcal{P}\Delta \vecu=-\Delta \vecu\in H \quad \textrm{for all}\quad \vecu\in\mathcal D(A).
\eeq

Define a bounded bilinear mapping $B:V\times V\to V'$ by
\beqs
\inprod{B(\vecu,\vecv),\vecw}_{V',V}=b(\vecu,\vecv,\vecw)\eqdef \int_\Omega ((\vecu\cdot \nabla) \vecv)\cdot \vecw\, \d\vecx  \textrm{ for all } \vecu,\vecv,\vecw\in V.
\eeqs 
In particular,
\beq\label{BDA}
B(\vecu,\vecv)=\mathcal{P}((\vecu\cdot \nabla) \vecv), \quad \textrm{for all}\quad \vecu,\vecv\in\mathcal D(A).
\eeq
There exists a constant $d_*>0$ such that
\beq\label{Bweak}
\|B(u,v)\|_{V'}\le d_* \|u\|\,\|v\|\quad \forall u,v\in V.
\eeq

Then the problems  \eqref{nse} and \eqref{ini} can be rewritten in the functional form as (see, e.g.,  \cite{LadyFlowbook69,CFbook,TemamAMSbook,TemamSIAMbook})
\begin{align}\label{fctnse}
&\frac{\d u(t)}{\d t} + Au(t) +B(u(t),u(t))=f(t) \quad \text{ in } V' \text{ on } (0,\infty),\\
\label{uzero} 
&u(0)=u^0\in H.
\end{align}
More specifically, we will deal with the following type of weak solutions.

\begin{definition}\label{lhdef}
Let $f\in L^2_{\rm loc}([0,\infty),H)$.
A \emph{Leray--Hopf weak solution} $u(t)$ of \eqref{fctnse} is a mapping from $[0,\infty)$ to $H$ such that 
\beqs
u\in C([0,\infty),H_{\rm w})\cap L^2_{\rm loc}([0,\infty),V),\quad u'\in L^{4/3}_{\rm loc}([0,\infty),V'),
\eeqs
and satisfies 
\beq\label{varform}
\ddt \inprod{u(t),v}+\doubleinprod{u(t),v}+b(u(t),u(t),v)=\inprod{f(t),v}
\eeq
in the distribution sense in $(0,\infty)$, for all $v\in V$, and the energy inequality
\beqs
\frac12|u(t)|^2+\int_{t_0}^t \|u(\tau)\|^2\d\tau\le \frac12|u(t_0)|^2+\int_{t_0}^t \langle f(\tau),u(\tau)\rangle \d\tau
\eeqs
holds for $t_0=0$ and almost all $t_0\in(0,\infty)$, and all $t\ge t_0$.  
Here, $H_{\rm w}$ denotes the topological vector space $H$ with the weak topology.
 
A \emph{regular} solution is a Leray--Hopf weak solution that belongs to $C([0,\infty),V)$.

If $t\mapsto u(T+t)$ is a Leray--Hopf weak  solution, respectively, regular solution, then we say $u$ is a Leray--Hopf weak solution, respectively, regular solution on $[T,\infty)$.
\end{definition}

It is well-known that a regular solution is unique among all Leray--Hopf weak solutions.
Also, if $u(t)$ is a Leray--Hopf weak solution then it is a Leray--Hopf weak solution on $[T,\infty)$ for almost every $T\in(0,\infty)$.

\subsection{The real Gevrey--Sobolev spaces}\label{realsp}
We review the spectral Sobolev and Gevrey spaces based on the Stokes operator. There exist a complete orthonormal basis $(\vecw_n)_{n=1}^\infty$ of $H$, and a  sequence of positive numbers $(\lambda_n)_{n=1}^\infty$ so that 
\beq \label{lambn}
0<\lambda_1\le \lambda_2\le \ldots \le \lambda_n\le \lambda_{n+1}\le \ldots,
\quad \lim_{n\to\infty}\lambda_n=\infty,
\eeq
\beq \label{Awn}
A\vecw_n=\lambda_n \vecw_n \quad\forall n\in\N.
\eeq 
Note that each $\lambda_n$ is an eigenvalue of $A$ and has finite multiplicity.
We denote by $(\Lambda_n)_{n=1}^\infty$ the  strictly increasing sequence of the above eigenvalues $\lambda_n$. Then we still have $\Lambda_n\to\infty$ as $n\to\infty$.
Denote $\spec(A)=\{\lambda_n:n\in\N\}=\{\Lambda_n:n\in\N\}$, which is the spectrum of $A$.
It is known that
$\spec(A)=\{|\vkL|^2: \veck\in\Z^3, \veck\ne \mathbf 0\}$, where $\begin{displaystyle}\vkL=2\pi\left (k_1/\ell_1,k_2/\ell_2,k_3/\ell_3\right)\end{displaystyle}$ for 
 $\veck=(k_1,k_2,k_3)\in \Z^3$. 
It follows that $\lambda_1=1$.

For $\Lambda\in\spec(A)$, we denote by  $R_\Lambda$ the orthogonal projection from $H$ to the eigenspace of $A$ corresponding to $\Lambda$,
and set 
$$P_\Lambda=\sum_{\lambda\in \spec(A),\lambda\le  \Lambda}R_\lambda.$$ 
Recall that each linear space $P_\Lambda H$ is finite dimensional.

For $\alpha,s,\sigma\in\R$, define, for $\vecu=\sum_{n=1}^\infty c_n\vecw_n\in H$ with $c_n=\inprod{\vecu,\vecw_n}$,
\beq\label{Aee}
A^\alpha \vecu= \sum_{n=1}^\infty c_n\lambda_n^\alpha \vecw_n,\quad
e^{s A} \vecu= \sum_{n=1}^\infty c_n e^{s \lambda_n} \vecw_n,\quad
e^{\sigma A^{1/2}} \vecu= \sum_{n=1}^\infty c_n e^{\sigma \sqrt{\lambda_n}} \vecw_n.
\eeq
For $L=A^\alpha,e^{sA},e^{\sigma A^{1/2}}$, the corresponding formula in \eqref{Aee} is, in fact, defined on the domain $\mathcal D(L)=\{\mathbf u\in H:L\mathbf u\in H\}$.
(Even when $\alpha<0$ or $s<0$ or $\sigma<0$, we do not use the domains $\mathcal  D(L)$ as functional spaces larger than $H$ in this paper.)
The formulas in \eqref{Aee} can be formulated with the Fourier series by
\beqs
A^\alpha \vecu=\sum_{\veck\ne \mathbf 0} |\vkL|^{2\alpha} \vecu_\veck e^{i\vkL\cdot 
\vecx},\  
e^{s A} \vecu=\sum_{\veck\ne \mathbf 0} e^{s |\vkL|^2} \vecu_\veck e^{i\vkL\cdot 
\vecx},\  
e^{\sigma A^{1/2}} \vecu=\sum_{\veck\ne \mathbf 0} e^{\sigma 
|\vkL|} \vecu_\veck e^{i\vkL\cdot 
\vecx}
\eeqs 
 for  $\vecu(\vecx)=\sum_{\veck\ne \mathbf 0}\vecu_\veck e^{i\vkL\cdot \vecx}\in H$  with constant vectors $\vecu_\veck\in \C^3$.

Let $\alpha,\sigma\ge 0$. The  Gevrey--Sobolev spaces are defined by
\beqs
G_{\alpha,\sigma}=\mathcal D(A^\alpha e^{\sigma A^{1/2}} )\eqdef \{ \vecu\in H: |\vecu|_{\alpha,\sigma}\eqdef |A^\alpha 
e^{\sigma A^{1/2}}\vecu|<\infty\}.
\eeqs
Each $G_{\alpha,\sigma}$ is a real Hilbert space with the inner product
\beq\label{Gprod}
\inprod{u,v}_{G_{\alpha,\sigma}}=\inprod{A^\alpha e^{\sigma A^{1/2}}u,A^\alpha e^{\sigma A^{1/2}}v} \text{ for }u,v\in G_{\alpha,\sigma}.
\eeq

Note that $G_{0,0}=\mathcal D(A^0)=H$,  $G_{1/2,0}=\mathcal D(A^{1/2})=V$, $G_{1,0}=\mathcal D(A)$. The inner product in \eqref{Gprod} when $\alpha=\sigma=0$, respectively, $\alpha=1/2$, $\sigma=0$, agrees with $\inprod{\cdot,\cdot}$ on $H$, respectively,  $\doubleinprod{\cdot,\cdot}$ on $V$ indicated at the beginning of this subsection and \eqref{Vprod}. Subsequently,  $\|\vecu\|=|\nabla \vecu|=|A^{1/2}\vecu|$ for $\vecu\in V$.
Also, the norms $|\cdot|_{\alpha,\sigma}$ are increasing in $\alpha$, $\sigma$, hence, the spaces $G_{\alpha,\sigma}$ are decreasing in $\alpha$, $\sigma$.

We recall well-known inequalities that will be used throughout.
There exist positive numbers $d_0(\alpha,\sigma)$,  for $\alpha\ge 0$ and $\sigma>0$, such that
\beq\label{als0}
|A^\alpha e^{-\sigma A}v| \le d_0(\alpha,\sigma) |v|\quad \forall v\in H,
\eeq
\beq\label{als1}
|A^\alpha e^{-\sigma A^{1/2}}v| \le d_0(2\alpha,\sigma) |v|\quad \forall v\in H,
\eeq 
\beq \label{als}
|A^\alpha v|\le d_0(2\alpha,\sigma) |e^{\sigma A^{1/2}}v|\quad  \forall v\in G_{0,\sigma}.
\eeq

For estimates of the Gevrey norms $|B(u,v)|_{\alpha,\sigma}$,  we recall a convenient inequality  from \cite[Lemma 2.1]{HM1} which itself generalizes the case $\alpha=0$ in \cite{FT-Gevrey}.  
There exists a constant $c_*>1$ such that if $\sigma\ge 0$ and $\alpha\ge 1/2$, then 
\beq\label{AalphaB} 
|B(u,v)|_{\alpha ,\sigma }\le c_*^\alpha  |u|_{\alpha +1/2,\sigma } \, |v|_{\alpha +1/2,\sigma}\quad \forall u,v\in G_{\alpha+1/2,\sigma}.
\eeq

\subsection{Complexified spaces and mappings}\label{complex}
We will use the idea of complexification, see, e.g., \cite[section 77]{Halmos1974}.
 We review the material from Section 4 of \cite{H6}.

\subsubsection{General complexification}\label{cmplx}

Let $X$ be a linear space over $\R$. Its complexification is the set of elements $z=x+iy$ with $x,y\in X$, and is denoted by $X_\C$.
The addition and scalar multiplication on $X_\C$ are  naturally defined by
\begin{align*}
z+z'=(x+x')+i(y+y'),\quad
cz=(ax-by)+i(bx+ay),
\end{align*}
for $z= x+iy,z'=x'+iy'\in X_\C$ with $x,x',y,y'\in X$, and  $c=a+ib\in \C$ with $a,b\in\R$.
Then $X_\C$ is a linear space over $\C$ and $X\subset X_\C$. 
For $z=x+iy\in X_\C$, with $x,y\in X$,  its conjugate is defined by 
$\overline z=x-iy=x+i(-y).$
 Then
\beqs
 \overline{cz}=\overline c\, {\overline z} \text{ for all } c\in\C, z\in X_\C.
\eeqs 

Suppose $(X,\langle\cdot,\cdot\rangle_X)$ is an inner product  space over $\R$. Then the complexification $X_\C$ is an inner product space over $\C$ with the corresponding inner product $\langle\cdot,\cdot\rangle_{X_\C}$ defined by
\beqs
\langle x+iy,x'+iy'\rangle_{X_\C}=\langle x,x'\rangle_X+\langle y,y'\rangle_X+i(\langle y,x'\rangle_X-\langle x,y'\rangle_X) \text{ for }x,x',y,y'\in X.
\eeqs
Denote by $\|\cdot\|_X$ and $\|\cdot\|_{X_\C}$ the norms on $X$ and $X_\C$ induced  from their respective inner products. Then   
\beqs
\|x+iy\|_{X_\C}=(\|x\|_X^2+\|y\|_X^2)^{1/2}\text{ and } \|\overline z\|_{X_\C}=\| z\|_{X_\C} \text{  for all $x,y\in X$ and $z\in X_\C$.}
\eeqs

Let $X$ and $Y$ be real linear spaces, and $X_\C$ and $Y_\C$ be their complexifications. Let $L$ be a linear mapping from $X$ to $Y$. The complexification $L$ is the mapping $L_\C:X_\C\to Y_\C$ defined by
\beqs
L_\C (x_1+ix_2)=Lx_1+i Lx_2 \text{ for  all } x_1,x_2\in X.
\eeqs
Clearly, $L_\C$ is the unique  linear extension of $L$ from $X$ to $X_\C$.

\begin{lemma}[ { \cite[Corollary 4.2]{H7}} ]\label{XYsame}
If $L:X\to Y$ is a bounded linear mapping between two real inner product spaces, then $L_\C$ is also a bounded linear mapping, and
\beqs
\|L_\C\|_{\mathcal B(X_\C,Y_\C)}=\|L\|_{\mathcal B(X,Y)}.
\eeqs 
Here, $\|\cdot\|_{\mathcal B(\cdot,\cdot)}$ denotes the norm of a bounded linear mapping.
\end{lemma}

\subsubsection{Complexification in the context of the NSE}\label{cmpNSE}

For $\alpha,\sigma\ge 0$, denote the complexification $(G_{\alpha,\sigma})_\C$ by $G_{\alpha,\sigma,\C}$; it is a complex Hilbert space and we abbreviate its norm $\|\cdot\|_{G_{\alpha,\sigma,\C}}$  by $|\cdot|_{\alpha,\sigma}$. In particular, $G_{1,0,\C}=(\mathcal D(A))_\C$, the norm $\|\cdot\|_{H_\C}$ is denoted by $|\cdot|$, and $\|\cdot\|_{V_\C}$ by $\|\cdot\|$.

\begin{definition}\label{ABC}
Considering the Stokes operator $A$ given by \eqref{ADA}, let $A_\C$ denote its complexification. Specifically,  $A_\C:G_{1,0,\C}\to H_\C$ is defined by 
\beqs
A_\C(u+iv)=Au+iAv \text{ for }u,v\in G_{1,0}.
\eeqs

Considering the bilinear form $B$ given by \eqref{BDA}, its complexification  is  $B_\C:G_{1,0,\C}\times G_{1,0,\C}\to H_\C$  defined by
\beqs
\begin{aligned}
&B_\C(u_1+iv_1,u_2+iv_2)=B(u_1,u_2)-B(v_1,v_2) +i ( B(u_1,v_2)+B(v_1,u_2))
\end{aligned}
\eeqs
for  $u_1,u_2,v_1,v_2\in G_{1,0}$.
\end{definition}

Then $A_\C$ is the unique linear mapping that extends $A$ from $G_{1,0}$ to $G_{1,0,\C}$.
Thanks to Lemma \ref{XYsame}, $A_\C$ is a bounded linear mapping from $G_{1,0,\C}$ to $H_\C$.
Similarly, $B_\C$ is the unique bilinear mapping that extends $B$ from $G_{1,0}\times G_{1,0}$ to $G_{1,0,\C}\times G_{1,0,\C}$.
Moreover,  $B_\C$ is a bounded bilinear mapping from $G_{1,0,\C}\times G_{1,0,\C}$ to $H_\C$. 
We  have, for all $w\in G_{1,0,\C}$ and $w_1,w_2\in G_{1,0,\C}$, that
\begin{equation}\label{BCbar}
\overline{A_\C w}=A_\C \overline{w} 
\text{ and }
\overline{B_\C(w_1,w_2)}=B_\C(\overline{w_1},\overline{w_2}).
\end{equation}

Let $(\lambda_n)_{n=1}^\infty$ and $(\vecw_n)_{n=1}^\infty$ be the ones in  \eqref{lambn} and \eqref{Awn}.
It is clear that  $(\vecw_n)_{n=1}^\infty$ is a complete orthonormal basis of $H_\C$ and 
\beqs
A_\C \vecw_n=\lambda_n \vecw_n \text{ for all } n\in\N.
\eeqs
Then the set of eigenvalues of $A_\C$ is exactly $\spec(A)$.
Moreover,  for any $\Lambda\in \spec(A)$, the eigenspace of $A_\C$ corresponding to $\Lambda$ is $(R_\Lambda H)_\C$.
For  $\Lambda\in \spec(A)$, define $R_{\Lambda,\C}$ to be the orthogonal projection from $H_\C$ to  the space $(R_\Lambda H)_\C$, and 
$$P_{\Lambda,\C}=\sum_{\lambda\in\spec(A),\lambda\le \Lambda}R_{\lambda,\C}.$$

For $\alpha,s,\sigma\in\R$ and $w=\sum_{n=1}^\infty c_n\vecw_n \in H_\C$ with $c_n=\inprod{w,\vecw_n}_{H_\C}\in \C$,  we define, in a similar manner to \eqref{Aee}, 
\beqs
A_\C^\alpha w=\sum_{n=1}^\infty \lambda_n^\alpha c_n\vecw_n,\quad
e^{s A_\C} w=\sum_{n=1}^\infty e^{s \lambda_n} c_n\vecw_n,\quad
e^{\sigma A_\C^{1/2}} w=\sum_{n=1}^\infty e^{\sigma \sqrt{\lambda_n}} c_n\vecw_n,
\eeqs
whenever the defined element belongs to $H_\C$. 
Then one has  
\beqs
R_{\Lambda,\C}=(R_{\Lambda})_\C,\quad 
P_{\Lambda,\C}=(P_{\Lambda})_\C,\quad
A_\C^\alpha e^{\sigma A_\C^{1/2}}=(A^\alpha e^{\sigma A^{1/2}})_\C,\quad 
e^{s A_\C}=(e^{s A})_\C. 
\eeqs
For $\alpha,\sigma\ge 0$ and $w\in G_{\alpha,\sigma,\C}$, we have $|w|_{\alpha,\sigma}=|w|_{\alpha,\sigma,\C}=|A_\C^\alpha e^{\sigma A_\C^{1/2}}w|$. 
The inequalities \eqref{als0}, \eqref{als1} and \eqref{als} are still valid for complexified spaces $H_\C$, $G_{0,\sigma,\C}$  and operator $A_\C$ with the same constants $d_0(\alpha,\sigma)$, $d_0(2\alpha,\sigma)$.

For $c\in \C$, we define the linear mapping  $A_\C+c{\mathbb I}:G_{1,0,\C}\to H_\C$ by 
$$(A_\C+c{\mathbb I})w=A_\C w +c w\text{  for  }w\in G_{1,0,\C}.$$
If $\omega\in\R$, then 
\beqs
|(A_\C+i\omega{\mathbb I})w|_{\alpha,\sigma}^2=|A_\C w|_{\alpha,\sigma}^2+\omega^2 |w|_{\alpha,\sigma}^2
\le (1+\omega^2) |w|_{\alpha+1,\sigma}^2
\text{ for $w\in G_{\alpha+1,\sigma,\C}$.}
\eeqs

For the existence and estimates of $(A_\C+i\omega{\mathbb I})^{-1}w$ we have the following.

\begin{lemma}[ {\cite[Lemma 4.4]{H6}} ]\label{Aioinv}
For any numbers $\alpha,\sigma\ge 0$ and $\omega\in \R$, one has $A_\C+i\omega{\mathbb I}$ is a bijective, bounded linear mapping from $G_{\alpha+1,\sigma,\C}$ to $G_{\alpha,\sigma,\C}$ with
\beq\label{AZA}
|(A_\C+i\omega{\mathbb I})^{-1}w|_{\alpha+1,\sigma}\le |w|_{\alpha,\sigma}\text{  for all }w\in  G_{\alpha,\sigma,\C}.
\eeq
\end{lemma}

In particular, applying Lemma \ref{Aioinv} to $\alpha=\sigma=0$ yields $A_\C+i\omega{\mathbb I}$ is a bijective, bounded linear mapping from $G_{1,0,\C}$ to $H_\C$.
We also have the following conjugation relation
\beq\label{Acinv}
\overline{(A_\C+i\omega{\mathbb I})^{-1}w}=(A_\C+\overline{i\omega}\,{\mathbb I})^{-1}\overline{w}.
\eeq

Finally, inequality \eqref{AalphaB} can be extended for $B_\C$, namely, for $\alpha\ge 1/2$, $\sigma\ge 0$, one has
 \beq\label{BCas}
|B_\C(w_1,w_2)|_{\alpha,\sigma}
\le \sqrt2 c_*^\alpha |w_1|_{\alpha+1/2,\sigma} |w_2|_{\alpha+1/2,\sigma} \quad\forall w_1,w_2\in G_{\alpha+1/2,\sigma,\C}.
\eeq

\section{Functions with subordinate variables and asymptotic expansions} \label{classtype}

We describe the asymptotic expansions that will be studied in details in this paper. They are new to the NSE but a less general form was already used in our previous work \cite{H5} for systems of ODEs in the Euclidean spaces.

In this paper, we make use of only  single-valued complex functions. To avoid any ambiguity we recall basic definitions and properties of elementary complex functions.
For $z\in\C$ and $t>0$, the exponential and power functions are defined by
\beq\label{ep}
\exp(z)=\sum_{k= 0}^\infty \frac{z^k}{k!}
\text{ and }
t^z=\exp(z\ln t).
\eeq
When $t=e=\exp(1)$ in \eqref{ep}, one has the usual identity
$e^z=\exp (z)$.

If $z=a+ib$ with $a,b\in\R$, then 
\beqs
t^z=t^a (\cos(b\ln t)+i\sin(b\ln t))\text{ and }|t^z|=t^a.
\eeqs
The standard properties of the power functions still hold, namely, 
\beqs
 t^{z_1}t^{z_2}=t^{z_1+z_2},\quad (t_1t_2)^z=t_1^z t_2^z,\quad (t^z)^m =t^{m z}= (t^m)^z,\quad \ddt (t^z)=zt^{z-1},
\eeqs
for any $t,t_1,t_2>0$, $z,z_1,z_2\in\C$, and $m\in\Z_+$.

Define the iterated exponential and logarithmic functions as follows:
\begin{align*} 
&E_0(t)=t \text{ for } t\in\R,\text{ and } E_{m+1}(t)=e^{E_m(t)}  \text{ for } m\in \Z_+, \ t\in \R,\\
& L_{-1}(t)=e^t,\quad  L_0(t)=t\text{ for } t\in\R,\text{ and }  L_{m+1}(t)= \ln( L_m(t)) \text{ for } m\in \Z_+,\ t>E_m(0).
\end{align*}
Explicitly, $ L_1(t)=\ln t$, $ L_2(t)=\ln\ln t$.
For $k\in \Z_+$, define
\beqs 
\widehat \LL_k(t)=( L_{-1}(t), L_{0}(t), L_1(t),\ldots, L_{k}(t))=(e^t,t,\ln t,\ln\ln t,\ldots, L_{k}(t)).
\eeqs 

It is clear, for $m\in\Z_+$, that 
\begin{align} 
&\text{$ L_m(t)$ is positive and increasing  for $t>E_m(0)$, }\label{Linc}\\
& L_m(E_{m}(0))=0, \quad 
 L_m(E_{m+1}(0))=1, \quad 
\lim_{t\to\infty}  L_m(t)=\infty.\label{Lone}
\end{align}
Also,
\beq  \label{LLk}
\lim_{t\to\infty} \frac{ L_k(t)^\lambda}{ L_m(t)}=0\text{ for all }k>m\ge -1 \text{ and } \lambda\in\R.
\eeq 

For $m\in \N$, the derivative of $ L_m(t)$ is
\beq\label{Lmderiv}
  L_m'(t)=\frac 1{t\prod_{k=1}^{m-1}  L_k (t)}=\frac 1{\prod_{k=0}^{m-1}  L_k (t)}. 
\eeq
With the use of the L'Hospital rule and \eqref{Lmderiv}, one can prove, by  induction,  that  it holds, for any $T\in\R$ and $c>0$, 
\beq\label{Lshift}
\lim_{t\to\infty}\frac{ L_m(T+ct)}{ L_m(t)}=
\begin{cases} 
c,& \text{ for }m=0,\\
1, & \text{ for }m\ge 1.
\end{cases}
\eeq
Consequently, if $T,T'>E_m(0)$ and $c,c'>0$, then there are numbers $C,C'>0$ such that
\beq\label{Lsh2}
C'\le\frac{ L_m(T+ct)}{ L_m(T'+c't)}\le C\text{ for all }t\ge 0.
\eeq

We recall a fundamental integral estimate that will be used throughout.

\begin{lemma}[{\cite[Lemma 2.5]{CaH3}}]\label{plnlem}
Let $m\in\Z_+$ and $\lambda>0$, $\gamma>0$  be given. For any number $T_*>E_m(0)$, there exists a number $C>0$ such that
\beq\label{iine2}
 \int_0^t e^{-\gamma (t-\tau)} L_m(T_*+\tau)^{-\lambda}\d\tau
 \le C  L_m(T_*+t)^{-\lambda} \quad\text{for all }t\ge 0.
\eeq
\end{lemma}
\begin{proof} We present a short proof here. For $t\ge 0$, let
$$g(t)= e^{\gamma t}L_m(T_*+t)^{-\lambda} \text{ and } h(t)=\int_0^t g(\tau)\d \tau.$$
Then $h\ge 0$, $g>0$ on $[0,\infty)$, they are $C^1$-functions and tend to infinity as $t\to\infty$. We have
$$\frac{h'(t)}{g'(t)}=\frac{g(t)}{\gamma g(t)- \lambda g(t)(\prod_{k=0}^{m}  L_k (t))^{-1}}
=\frac{1}{\gamma - \lambda (\prod_{k=0}^{m}  L_k (t))^{-1}}$$
which converges to $1/\gamma$ as $t\to\infty$. 
Applying the L'Hospital Rule gives $\lim_{t\to\infty}h(t)/g(t)=1/\gamma$. 
Combining this fact with the existence of the maximum of $h/g$ on any bounded, closed subinterval of $[0,\infty)$, we obtain the estimate \eqref{iine2} for all $t\ge 0$.
\end{proof}

\subsection{Functions with subordinate variables}
Let $\K=\R$ or $\C$. For $k\ge -1$,
\beq
\label{azvec} 
z=(z_{-1},z_0,z_1,\ldots,z_k)\in (0,\infty)^{k+2}
\text{ and } 
\beta=(\beta_{-1},\beta_0,\beta_1,\ldots,\beta_k)\in \K^{k+2},
\eeq
 define 
 $\begin{displaystyle} z^{\beta}=\prod_{j=-1}^k z_j^{\beta_j}
 \end{displaystyle}.$
Similarly, for $\ell\ge 1$,
\beq
\label{zevec} 
\zeta=(\zeta_1,\ldots,\zeta_\ell)\in (0,\infty)^\ell
\text{ and } 
\gamma=(\gamma_1,\ldots,\gamma_\ell)\in \K^\ell,
\eeq
define
$\begin{displaystyle}
  \zeta^{\gamma}=\prod_{j=1}^\ell \zeta_j^{\gamma_j}
 \end{displaystyle}.$

For $\mu\in\R$, $m,k\in\Z$ with  $k\ge m\ge -1$, denote by $\mathcal E_\K(m,k,\mu)$ the set of vectors $\beta$ in \eqref{azvec} 
 such that
 \beqs
\Re(\beta_j)=0 \text{ for $-1\le j<m$   and  } \Re(\beta_m)=\mu.
\eeqs 
Particularly, $\mathcal E_\R(m,k,\mu)$ is the set of vectors $\beta=(\beta_{-1},\beta_0,\ldots,\beta_k)\in \R^{k+2}$   such that
 $$\beta_{-1}=\ldots=\beta_{m-1}=0 \text{ and }\beta_m=\mu.$$
For example, when $m=-1$, $k\ge -1$, $\mu=0$, the set 
\beq \label{Eminus}
\mathcal E_\K(-1,k,0)\text{  is  the collection of vectors $\beta$ in \eqref{azvec} with $\Re (\beta_{-1})=0$. }
\eeq 

Let $k\ge m\ge -1$, $\mu\in\R$, and $\beta\in\mathcal E_\K(m,k,\mu)$.
Using \eqref{LLk}, one can verify that, see, e.g., equation (3.14) in \cite{H5}, 
\beq\label{LLo}
\lim_{t\to\infty} \frac{\widehat\LL_{k}(t)^\beta}{ L_{m}(t)^{\mu+\delta}}=0 \quad\text{ for any }\delta>0.
\eeq

\begin{definition}\label{Hclass}
 Let $\K$ be $\R$ or $\C$,  and  $X$ be a linear space over $\K$.

\begin{enumerate}[label=\tnum]
\item \label{Hi} For $k\ge -1$, $\ell\ge 1$, define $\widehat{\mathscr P}(k,\ell,X)$ to be the set of functions of the form 
\beq\label{pzedef} 
p(z,\zeta)=\sum_{(\beta,\gamma)\in S}  z^{\beta}\zeta^\gamma \xi_{\beta,\gamma}\text{ for }z\in (0,\infty)^{k+2}\text{ and }\zeta\in(0,\infty)^\ell,
\eeq 
where $S$ is some finite subset of $\K^{k+2}\times \K^\ell$, and each $\xi_{\beta,\gamma}$ belongs to $X$.

\item \label{Hii} Let $k\ge m\ge -1$, $\ell\ge 1$,  and $\mu\in\R$. 
Define $\widehat{\mathscr P}_{m}(k,\ell,\mu,X)$ to be the set of functions $p\in\widehat{\mathscr P}(k,\ell,X)$ of the form \eqref{pzedef} such that  $\beta\in \mathcal E_\K(m,k,\mu)$  for each 
$(\beta,\gamma)\in S$.

\item \label{Hiii} Let ${\mathscr P}(k,X)$, respectively,  ${\mathscr P}_{m}(k,\mu,X)$ be defined the same way as $\widehat{\mathscr P}(k,\ell,X)$, respectively,  $\widehat{\mathscr P}_{m}(k,\ell,\mu,X)$, but for functions $p=p(z)$ only, that is, without variable $\zeta$, dimension $\ell$, vectors $\gamma$, and the sets $(0,\infty)^\ell$, $\K^\ell$.
Denote
 \beqs
 \widehat{\mathscr P}(k,0,X)=\mathscr P(k,X)\text{ and }
  \widehat{\mathscr P}_m(k,0,\mu,X)=\mathscr P_m(k,\mu,X). 
 \eeqs
 \end{enumerate}
 \end{definition}
 
 For a function $p(z,\zeta)$ in \eqref{pzedef}, we call $\zeta_1,\zeta_2,\ldots$ the \textit{subordinate variables}.
 The reason for this terminology is that they will be used with the substitions \eqref{Zk} and \eqref{Yk} in the asymptotic expansion \eqref{fiter}.
 
The following remarks on Definition \ref{Hclass} are in order.

\begin{enumerate}[label=(\alph*)]
\item Given any number $\ell\in \N$. For $p\in \mathscr P(k,X)$, by considering $S=S\times\{\mathbf 0_\ell\}\subset \K^{k+2}\times \K^\ell$ for $S\subset \K^{k+2}$ and $z^\beta=z^\beta \zeta^{\mathbf 0_\ell}$, we have $p=p(z,\zeta)\in  \widehat{\mathscr P}(k,\ell,X)$. Thus, we have
\beqs
\mathscr P(k,X)\subset  \widehat{\mathscr P}(k,\ell,X),\text{ and, similarly, }
\mathscr P_m(k,\mu,X)\subset  \widehat{\mathscr P}_m(k,\ell,\mu,X).
\eeqs

 \item  By considering 
 \beq \label{zbtil}
 \text{variable $\tilde z=(z,\zeta)\in (0,\infty)^{k+\ell+2}$ and power vector $\tilde\beta=(\beta,\gamma)\in \K^{k+\ell+2}$,}
 \eeq 
 we have the identities 
\beq \label{samespace}
\widehat{\mathscr P}(k,\ell,X)= {\mathscr P}(k+\ell,X)\text{ and }
 \widehat{\mathscr P}_{m}(k,\ell,\mu,X)={\mathscr P}_{m}(k+\ell,\mu,X).
 \eeq  
 However, because the variables $z$ and $\zeta$ will play different roles, the identities in \eqref{samespace} will seldom be used, only at some technical steps.
 
 \item\label{Cb} Each $\widehat{\mathscr P}(k,\ell,X)$, $\widehat{\mathscr P}_{m}(k,\ell,\mu,X)$ is a linear space over $\K$.

 \item\label{Cc} Let $k'\ge k\ge -1$ and $\ell'\ge \ell\ge 1$. For the power vectors $\beta\in \K^{k+2}$ and $\gamma\in\K^\ell$ as in \eqref{azvec} and \eqref{zevec}, it is standard to consider them as
 \beq \label{RimR}
 \beta=(\beta_{-1},\ldots,\beta_k,0,\ldots,0)\in \K^{k'+2}\text{ and }
 \gamma=(\gamma_1,\ldots,\gamma_\ell,0,\ldots,0)\in\K^{\ell'}.
 \eeq 
 Thus, we will widely use  the embeddings
   \beq \label{PPemd}
   \widehat{\mathscr P}(k,\ell,X)\subset \widehat{\mathscr P}(k',\ell',X)\text{ and }
   \widehat{\mathscr P}_m(k,\ell,\mu,X)\subset \widehat{\mathscr P}_m(k',\ell',\mu,X).
   \eeq
 Clearly, \eqref{PPemd} also holds true for $\ell'\ge \ell\ge 0$.
  
 \item\label{Cd} For $k\ge m\ge -1$, $\ell\ge 0$ and $\mu\in\R$, one has
 \beq\label{qpequiv}
 \begin{aligned}
& q\in  \widehat{\mathscr P}_{m}(k,\ell,\mu,X) \text{ if and only if }\\
 & \exists p\in  \widehat{\mathscr P}_{m}(k,\ell,0,X): q(z,\zeta)=p(z,\zeta) z_m^{\mu} \text{ with $z$ as in \eqref{azvec}}. 
 \end{aligned}
 \eeq
 
 \item \label{Ce} For any $k\ge m>m'\ge -1$, $\ell\ge 0$ and $\mu\in\R$, one has
\beq\label{Pmmz}
\widehat{\mathscr P}_m(k,\ell,\mu,X)\subset \widehat{\mathscr P}_{m'}(k,\ell,0,X) . 
\eeq

\item\label{Cf} 
If $X$ is a normed space and $p\in\widehat{\mathscr P}_{m}(k,\ell,\mu,X)$, then $\partial p/\partial \zeta_j\in\widehat{\mathscr P}_{m}(k,\ell,\mu,X)$ for all  $1\le j\le \ell$.   
\end{enumerate}

Since our results for the NSE involve real-valued functions, the following counterpart of Definition \ref{Hclass} is needed.

\begin{definition}\label{realH}
 Let $X$ be a linear space over $\R$, and $X_\C$ be its complexification.

\begin{enumerate}[label=\tnum]
\item Define $\widehat{\mathscr P}(k,\ell,X_\C,X)$ to be set of functions of the form 
\beq\label{Rpzdef} 
p(z,\zeta)=\sum_{(\beta,\gamma)\in S}  z^\beta\zeta^\gamma\xi_{\beta,\gamma}\text{ for }z\in (0,\infty)^{k+2}\text{ and } \zeta\in(0,\infty)^\ell,
\eeq 
where $S$ is a finite subset of $\C^{k+2}\times \C^\ell$ that preserves the conjugation,
and each $\xi_{\beta,\gamma}$ belongs to $X_\C$, with 
\beq\label{xiconj}
\xi_{\overline \beta,\overline \gamma}=\overline{\xi_{\beta,\gamma} }\quad\text{ for all } (\beta,\gamma)\in S.
\eeq

\item\label{realii} Define $\widehat{\mathscr P}_m(k,\ell,\mu,X_\C,X)$ to be set of functions in $\widehat{\mathscr P}(k,\ell,X_\C,X)$ with the restriction that the set $S$ in \eqref{Rpzdef} is also a subset of $\mathcal E_\C(m,k,\mu)$.

\item Let  ${\mathscr P}(k,X_\C,X)$, respectively,  ${\mathscr P}_m(k,\mu,X_\C,X)$,  be defined by the same way as $\widehat{\mathscr P}(k,\ell,X_\C,X)$, respectively,  $\widehat{\mathscr P}_m(k,\ell,\mu,X_\C,X)$,  but for functions $p=p(z)$ only.
Denote
\beqs
 \widehat{\mathscr P}(k,0,X_\C,X)=\mathscr P(k,X_\C,X)\text{ and }
  \widehat{\mathscr P}_m(k,0,\mu,X_\C,X)=\mathscr P_m(k,\mu,X_\C,X). 
 \eeqs
\end{enumerate}
\end{definition}

Based on the fact 
$z^{\overline \beta}\zeta^{\overline \gamma}=\overline{z^\beta\zeta^\gamma}$  
and the condition \eqref{xiconj}, the summation in \eqref{Rpzdef} yields that each function $p\in \widehat{\mathscr P}(k,\ell,X_\C,X)$ is $X$-valued.
In fact, we can rewrite \eqref{Rpzdef} as 
\beq\label{symsum} 
p(z,\zeta)=\frac12\sum_{(\beta,\gamma)\in S}  (z^\beta\zeta^\gamma\xi_{\beta,\gamma}+z^{\overline\beta}\zeta^{\overline\gamma}\,\overline{\xi_{\beta,\gamma}}).
\eeq 

Same as for \eqref{samespace}, we have
\beq \label{same2}
\widehat{\mathscr P}(k,\ell,X_\C,X)={\mathscr P}(k+\ell,X_\C,X)\text{ and }
 \widehat{\mathscr P}_{m}(k,\ell,\mu,X_\C,X)={\mathscr P}_{m}(k+\ell,\mu,X_\C,X).
 \eeq  
Observe that $\widehat{\mathscr P}(k,\ell,X_\C,X)$  and  $\widehat{\mathscr P}_m(k,\ell,X_\C,X)$ are additive subgroups of $\widehat{\mathscr P}(k,\ell,X_\C)$, but not linear spaces over $\C$.

\subsection{Substitutions for the subordinate variables}\label{subor}
We will use the lexicographic order for the power vectors in $\R^n$, for example,  in \eqref{bposcond} and \eqref{betaplus} below.
 
\begin{definition}\label{plusclass}
 Let $k\ge m\ge 0$. 
 
 \begin{enumerate}[label=\tnum]
\item\label{plusi} For $\ell\ge 1$, define $\widehat{\mathscr P}_m^+(k,\ell)$ to a the set of functions  $p\in \widehat{\mathscr P}_m(k,\ell,0,\C,\R)$  of the form 
\beq\label{pdecomp}
p(z,\zeta)=p_*(\zeta)+q(z,\zeta),
\eeq 
where 
\begin{enumerate}[label=\rnum]
\item \label{plusa}
the function $p_*(\zeta)$ is a finite sum of $\zeta^\gamma c_\gamma$, for $\zeta\in(0,\infty)^\ell$, with $\gamma\in \R^\ell$ and $c_\gamma\in \R$ so that 
$p_*(\mathbf 1_\ell)=1$, and 

\item \label{plusb} the function $q(z,\zeta)$ has the formula on the right-hand side of \eqref{Rpzdef}, belongs to the class $\widehat{\mathscr P}_m(k,\ell,0,\C,\R)$ as defined in Definition \ref{realH}\ref{realii}, and additionally satisfies
\beq \label{bposcond}
\Re\beta<\mathbf 0_{k+2}\text{ and }\beta_{-1}=0
\eeq 
 for all power vectors $\beta$ in \eqref{Rpzdef} with the form \eqref{azvec}. 
\end{enumerate}

\item Define ${\mathscr P}_m^+(k)$ in the same way as part \ref{plusi} but for functions $p=p(z)$ only. 
Specifically,
\beq\label{pcomp2} 
p(z)=1+q(z),
\eeq 
where $q\in \mathscr P_m(k,0,\C,\R)$ has the same properties as in the requirement \ref{plusb} of part \ref{plusi} above without the presence of $\zeta$. 
Denote $\widehat{\mathscr P}_m^+(k,0)={\mathscr P}_m^+(k)$.
\end{enumerate}
\end{definition}

In fact, condition \eqref{bposcond} explicitly is
\beq\label{betaplus}
\beta_{-1}=\Re(\beta_0)=\Re(\beta_1)=\ldots=\Re(\beta_m)=0\text{ and } \Re(\beta_{m+1},\ldots,\beta_k)<\mathbf 0_{k-m}.
\eeq
We emphasize, for $k\ge m\ge 0$ and $\ell\ge 0$, that
\beq\label{PPrel}
 \widehat{\mathscr P}_m^+(k,\ell)\subset \widehat{\mathscr P}_m(k,\ell,0,\C,\R)\subset \widehat{\mathscr P}_0(k,\ell,0,\C,\R).
\eeq
Thanks to the first condition in \eqref{bposcond}, the function $q(z,\zeta)$ in \eqref{pdecomp} satisfies
\beq\label{limq}
\lim_{t\to\infty} q(\widehat \LL_k(t),\zeta)=0 \text{ uniformly in $\zeta$ in any compact subsets of $(0,\infty)^\ell$.}
\eeq
Consequently,
\beq\label{limp}
\lim_{\substack{t\to\infty\\ \zeta_1,\ldots,\zeta_{\ell-1}\to 1}} p(\widehat \LL_k(t),\zeta)=p_*(\mathbf 1_\ell)=1.
\eeq
Similarly, $q(z)$ and $p(z)$ in \eqref{pcomp2} satisfy the same limits as in \eqref{limq} and \eqref{limp}, respectively,  without $\zeta$.

\begin{definition}\label{Zsys}
Let $m\in\Z_+$. Suppose the set  $\mathcal K=\{1,2,\ldots,K^*\}$ for some number $K^*\in \N$, or $\mathcal K=\N$;
 the numbers $s_k\in\N$, for $k\in\mathcal K$, are increasing in $k$; and functions $Z_k$, for $k\in\mathcal K$, satisfy
\beq \label{Zs}
Z_k\in \widehat{\mathscr P}_m^+(s_k,k-1).
\eeq 
Define the functions $\mathcal Y_k(t)$, for $k\in\mathcal K$, recursively as follows
\begin{align}
 \mathcal Y_1(t)&=Z_1(\widehat \LL_{s_1}(t)), \label{Y1}\\
\mathcal Y_k(t)&=Z_k(\widehat \LL_{s_k}(t),\mathcal Y_1(t),\mathcal Y_2(t),\ldots,\mathcal Y_{k-1}(t)) \text{ for }k\ge 2.\label{Yk}
\end{align}
For $k\in\mathcal K$, let 
\beq \label{hatY}
\widehat{\mathcal Y}_k(t)=(\mathcal Y_1(t),\mathcal Y_2(t),\ldots,\mathcal Y_k(t)).
\eeq 
Denote by $\mathscr U(m)$ the set of all of the above triples $(\mathcal K,(s_k)_{k\in\mathcal K},(Z_k)_{k\in\mathcal K})$.
\end{definition}

Since $Z_1\in \widehat{\mathscr P}_m^+(s_1,0) ={\mathscr P}_m^+(s_1)$, formula \eqref{Y1} is valid.
Note that the numbers $s_k$ are not necessarily strictly increasing.
If the set $\mathcal K$ is finite, the functions $Z_k$, for $k\in\mathcal K$,  depend only on finitely many variables $z_1,...,z_{s_{K^*}}$ and $\zeta_1,...,\zeta_{K^*}$. 
Roughly speaking, the functions $\mathcal Y_k(t)$ are obtained by  the substitutions
\begin{equation}\label{Zk}
\zeta_1=Z_1(z_{-1},z_0,z_1,\ldots,z_{s_1}),\quad 
\zeta_k=Z_k(z_{-1},z_0,z_1,\ldots,z_{s_k},\zeta_1,\zeta_2,\ldots,\zeta_{k-1}),
\end{equation}
and then evaluating $\zeta_k$ when all $z_j=L_j(t)$. 

Thanks to \eqref{limp}, one observes that
\beq\label{Zlim}
\lim_{t\to \infty} Z_1(\widehat \LL_{s_1}(t))=1
\text{ and }
\lim_{\substack{t\to \infty,\\ \zeta_1,\ldots,\zeta_{k-1}\to 1}} Z_k(\widehat \LL_{s_k}(t),\zeta)=1.
\eeq
These limits and \eqref{Y1}, \eqref{Yk}  imply recursively that 
\beq\label{limYk}
\lim_{t\to\infty} \mathcal Y_k(t)=1\text{ for all }k\in\mathcal K.
\eeq
In fact, we can obtain recursively for each $k\in\mathcal K$ that $\mathcal Y_k(t)$ is a $C^\infty$-function from $[T_k,\infty)$ to $\R$ for some large $T_k>0$.
Thanks to \eqref{limYk}, we can assume that $T_k$ is sufficiently large such that  $\mathcal Y_k(t)\ge 1/2$ for all $t\ge T_k$.

For convenience, whenever $\mathcal Y_0(t)$ or $\widehat{\mathcal Y}_0(t)$ appears in a computation, we just consider it void. With this convention,   we have  from  \eqref{Yk} and \eqref{hatY} that, for $k\in\mathcal K$,
\beq\label{Ykb}
\mathcal Y_k(t)=Z_k(\widehat \LL_{s_k}(t),\widehat{\mathcal Y}_{k-1}(t)).
\eeq
Using the embeddings \eqref{RimR}  we can also rewrite \eqref{Ykb} as
\beqs
\mathcal Y_k(t)=Z_k(\widehat \LL_{s'}(t),\widehat{\mathcal Y}_{k'}(t))
\text{ for any $s'\ge s_k$, $k'\in\mathcal K$ with $k'\ge k-1$ and all sufficiently large $t$.}
\eeqs

It is worth mentioning that the second condition in \eqref{bposcond} is imposed for a technical reason. It prevents the integration by parts in the proof of Theorem \ref{Fode} from being performed repeatedly without having a favorable estimate. 

\subsection{Asymptotic expansions}\label{exsec}

Now, we define the new asymptotic expansions which contain a system of subordinate variables.

\begin{definition}\label{Hexpand}
Let $\K$ be $\R$ or $\C$, and $(X,\|\cdot\|_X)$ be a normed space over $\K$. Suppose $g$ is a function from $(T,\infty)$ to $X$ for some $T\in\R$. 
Let  $m_*\in \Z_+$, $(\mathcal K,(s_k)_{k\in\mathcal K},(Z_k)_{k\in\mathcal K})\in \mathscr U(m_*)$ and $\widehat{\mathcal Y}_k(t)$ be as in Definition \ref{Zsys}.
Let $(\gamma_k)_{k=1}^\infty$ be a divergent, strictly increasing sequence of positive numbers, and $(M_k)_{k=1}^\infty$, $(\widetilde M_k)_{k=1}^\infty$ be increasing  sequences in $\N$ with $M_k\ge m_*$, $\widetilde M_k\in\mathcal K$ and  $M_k\ge s_{\widetilde M_k}$ for all $k\in \N$.

We say the function $g(t)$ has an asymptotic expansion
\beq\label{fiter}
g(t)\sim \sum_{k=1}^\infty p_k\left(\widehat{\LL}_{M_k}(t),\widehat{\mathcal Y}_{\widetilde M_k}(t)\right), \text{ where $p_k\in \widehat{\mathscr P}_{m_*}(M_k,\widetilde M_k,-\gamma_k,X)$ for $k\in\N$, }
\eeq
if, for each $N\in\N$, there is some number $\mu>\gamma_N$ such that
\beq\label{remain}
\left\|g(t) - \sum_{k=1}^N p_k\left(\widehat{\LL}_{M_k}(t),\widehat{\mathcal Y}_{\widetilde M_k}(t)\right)\right\|_X=\bigo( L_{m_*}(t)^{-\mu}).
\eeq
\end{definition}

By using the equivalence \eqref{qpequiv}, we can rewrite the asymptotic expansion  \eqref{fiter} as
\beq\label{expan3}
g(t)\sim \sum_{k=1}^\infty \widehat p_k\left(\widehat{\LL}_{M_k}(t),\widehat{\mathcal Y}_{\widetilde M_k}(t)\right) L_{m_*}(t)^{-\gamma_k}, \text{ where $\widehat p_k\in \mathscr P_{m_*}(M_k,\widetilde M_k,0,X)$ for $k\in\N$. }
\eeq

Suppose $\widehat p_k(z,\zeta)$ is given as a finite sum in \eqref{pzedef} with all $\xi_{\beta,\gamma}\ne 0$.
Then each term $a_{k,\beta,\gamma}(t)\eqdef \widehat{\LL}_{M_k}(t)^\beta \widehat{\mathcal Y}_{\widetilde M_k}(t)^\gamma\xi_{\beta,\gamma}$  
of  $\widehat p_k(\widehat{\LL}_{M_k}(t),\widehat{\mathcal Y}_{\widetilde M_k}(t)) $ in \eqref{expan3} satisfies, thanks to \eqref{limYk},
$$
\lim_{t\to\infty} \frac{\|a_{k,\beta,\gamma}(t)\|_X}{\widehat{\LL}_{M_k}(t)^{\Re\beta}}=\mathbf 1_{\widetilde M_k}^{\Re\gamma} \|\xi_{\beta,\gamma}\|_X=\|\xi_{\beta,\gamma}\|_X.
$$
Thus, one has, for sufficiently large $t$,
$$
\frac12 \widehat{\LL}_{M_k}(t)^{\Re\beta} \|\xi_{\beta,\gamma}\|_X 
\le \|a_{k,\beta,\gamma}(t)\|_X
\le 2 \widehat{\LL}_{M_k}(t)^{\Re\beta} \|\xi_{\beta,\gamma}\|_X.
$$
With $\Re\beta\in \mathcal E_\R(m_*,M_k,0)$, this shows that $a_{k,\beta,\gamma}(t)$  does not contribute  any extra $L_{m_*}(t)^r$, with some $r\in\R$, to the decaying mode $ L_{m_*}(t)^{-\gamma_k}$ in \eqref{expan3}. 
This fact provides a justification for the requirement \eqref{remain} and, hence, the definition of the asymptotic expansion \eqref{fiter}.

\section{Main results}\label{results}

We describe the main results first using the complexified Gevrey--Sobolev spaces $G_{\alpha,\sigma,\C}$ in subsection \ref{Csec}, and then using the real  spaces $G_{\alpha,\sigma}$ in subsection \ref{Rsec}

\begin{assumption}\label{A1} Hereafter, the function $f$ in \eqref{fctnse} belongs to $L^\infty_{\rm loc}([0,\infty),H)$.
\end{assumption}

Under Assumption \ref{A1}, for any $u^0\in H$, there exists a Leray--Hopf weak solution $u(t)$ of \eqref{fctnse} and \eqref{uzero}, see e.g. \cite{FMRTbook}. 
We recall a  result on the eventual regularity and asymptotic estimates for $u(t)$.

\begin{theorem}[{\cite[Theorem 3.4]{CaH2}}]\label{Fthm2}
Let $F$ be a continuous, decreasing, non-negative function on $[0,\infty)$
 that satisfies
$\lim_{t\to\infty} F(t)=0$.
Suppose there exist $\sigma\ge 0$ and $\alpha\ge 1/2$ such that
\beq\label{falphaonly}
|f(t)|_{\alpha,\sigma}=\mathcal O(F(t)).
\eeq

Let $u(t)$ be a Leray--Hopf weak solution of \eqref{fctnse}.  
Then  there exists  $\hat{T}>0$ 
such that $u(t)$ is a regular solution of \eqref{fctnse} on $[\hat{T},\infty)$, and $u(\hat T+t)\in G_{\alpha+1-\varep}$ for any $t\ge 0$ and $\varep\in(0,1)$.
If, in addition, $F$ satisfies 
 \begin{enumerate}[label=\tnum]
  \item  there exist $k_0>0$ and $D_1>0$ such that
$  e^{-k_0 t}\le D_1 F(t)$ for all $t\ge 0$,
 and
\item \label{Frii} for any $a\in(0,1)$, there exists $D_2=D_{2,a}>0$ such that 
$F(at)\le D_2 F(t)$ for all $t\ge 0$,
 \end{enumerate}
 then for any $\varep\in (0,1)$, there exists  $C>0$ such that
 \beq\label{us0}
 |u(\hat{T}+t)|_{\alpha+1-\varep,\sigma} \le CF(t)\text{ for all }  t\ge 0.
 \eeq 
\end{theorem}

In particular, estimate \eqref{us0} holds for $F(t)=L_m(T_*+t)^{-\mu}$ for $m\ge 0$, $\mu>0$, $t\ge 0$ and a sufficiently large $T_*>0$. Indeed, the requirement \ref{Frii} is met in this case thanks to \eqref{Lsh2}. 

Next, we define the linear transformations $\mathcal M_j$, $\mathcal R$,  $\mathcal Z_{A_\C}$ which are important to our exposition.

\begin{definition}\label{defMR}
Let $X$ be a linear space over $\K=\R$ or $\C$.
Given integers $k\ge -1$, $\ell\ge 0$, let $p\in \widehat{\mathscr P}(k,\ell,X)$ be given by \eqref{pzedef} with  $z\in(0,\infty)^{k+2}$ and $\beta\in \K^{k+2}$ as in \eqref{azvec}. 
\begin{enumerate}[label=\tnum]
\item Define, for $j=-1,0,\ldots,k$, the function $\mathcal M_jp:(0,\infty)^{k+2}\times(0,\infty)^\ell\to X$ by 
\beq\label{MM}
(\mathcal M_jp)(z,\zeta)=\sum_{(\beta,\gamma)\in S} \beta_j z^\beta\zeta^\gamma \xi_{\beta,\gamma}.
\eeq 

\item In the case $k\ge 0$, define the function $ \mathcal R  p:(0,\infty)^{k+2}\times(0,\infty)^\ell\to X$ by 
 \beq\label{chiz}
 (\mathcal R p)(z,\zeta)=
  \sum_{j=0}^k z_0^{-1}z_1^{-1}\ldots z_{j}^{-1}(\mathcal M_j p)(z,\zeta).
 \eeq 
 
 \item  By mapping $p\mapsto \mathcal M_j p$ and, respectively,  $p\mapsto \mathcal Rp$, one defines linear transformation $\mathcal M_j$ on $\mathscr P(k,X)$ for $-1\le j\le k$, and, respectively, 
  linear transformation $\mathcal R$ on $\mathscr P(k,X)$ for $k\ge 0$.
 \end{enumerate}
\end{definition}

In particular, one has from \eqref{MM} that
\beqs
\mathcal M_{-1}p(z,\zeta)=\sum_{(\alpha,\gamma)\in S} \alpha_{-1} z^\beta\zeta^\gamma \xi_{\beta,\gamma}
\quad \text{and}\quad 
\mathcal M_0p(z,\zeta)=\sum_{(\alpha,\gamma)\in S} \alpha_0 z^\beta\zeta^\gamma \xi_{\beta,\gamma} .
\eeqs
Combining \eqref{chiz} with \eqref{MM}, we can rewrite \eqref{chiz} explicitly as
\beq\label{chiz2}
 (\mathcal R p)(z,\zeta)=
  \sum_{(\beta,\gamma)\in S} \sum_{j=0}^k z_0^{-1}z_1^{-1}\ldots z_{j}^{-1}\beta_j z^\beta\zeta^\gamma \xi_{\beta,\gamma}.
 \eeq 

The power vectors $\beta$ of $z$  in \eqref{MM} for $\mathcal M_jp(z,\zeta)$ are the same as those that appear in \eqref{pzedef} for $p(z,\zeta)$. 
This and the explicit formula \eqref{chiz2} yield the following facts. 
\begin{enumerate}[label=\rnum]
\item \label{R0} For $k\ge m\ge 0$, $\ell\ge 0$ and $\mu\in\R$,   if $p$ is in $\widehat{\mathscr P}_{m}(k,\ell,\mu,X)$,  then so are all $\mathcal M_jp$, for $-1\le j\le k$.

\item\label{R1} 
 $\mathcal R p(z,\zeta)$ has the same powers of  $z_{-1}$ as $p(z,\zeta)$.
 
\item \label{R2}   
If $p\in\widehat{\mathscr P}_0(k,\ell,\mu,X)$, then $\mathcal R p\in\widehat{\mathscr P}_0(k,\ell,\mu-1,X)$.
\end{enumerate}
 
\begin{definition}\label{defZA}
Given  integers $k\ge -1$ and $\ell\ge 0$.
\begin{enumerate}[label=\tnum]
\item Let $p\in \widehat{\mathscr P}_{-1}(k,\ell,0,H_\C)$ be given by \eqref{pzedef} with   $\beta\in \C^{k+2}$ as in \eqref{azvec}. 
Define the function $\mathcal Z_{A_\C}p:(0,\infty)^{k+2}\times(0,\infty)^\ell\to G_{1,0,\C}$ by 
 \beq\label{ZAp}
 (\mathcal Z_{A_\C}p)(z,\zeta)=\sum_{(\beta,\zeta)\in S}  z^{\beta}\zeta^\gamma(A_\C+\beta_{-1}{\mathbb I})^{-1} \xi_{\beta,\gamma}.
 \eeq
 
 \item
  By  mapping  $p\mapsto \mathcal Z_{A_\C} p$, one defines the linear transformation   $\mathcal Z_{A_\C}$ on $\widehat{\mathscr P}_{-1}(k,\ell,0,H_\C)$.
\end{enumerate}
\end{definition}

Note that each $\beta=(\beta_{-1},\beta_0,\ldots,\beta_k)$ in \eqref{ZAp} belongs to $\mathcal E_\C(-1,k,0)$, which, by \eqref{Eminus}, yields $\Re(\beta_{-1})=0$. Therefore, $(A_\C+\beta_{-1}{\mathbb I})^{-1}\xi_{\beta,\gamma}$ exists and belongs to $G_{1,0,\C}$  thanks to Lemma \ref{Aioinv}.
If $\beta_{-1}=0$ for all $(\beta,\gamma)\in S$ in \eqref{ZAp}, then $\mathcal Z_{A_\C}p=A_\C^{-1}p$.
Moreover, in the case $p\in \mathscr P_{-1}(k,\ell,0,H)$, which corresponds to $\K=\R$, then $\beta_{-1}=0$, $\beta_0,\ldots,\beta_k\in\R$  and $\xi_{\beta,\gamma}\in H$ for all $(\beta,\gamma)\in S$ in \eqref{ZAp}, hence, $\mathcal Z_{A_\C}p=A^{-1}p$.

\subsection{Statements with complexified spaces}\label{Csec}
Let  $m_*\in \Z_+$, $(\mathcal K,(s_k)_{k\in\mathcal K},(Z_k)_{k\in\mathcal K})$ belong to $\mathscr U(m_*)$ and $\widehat{\mathcal Y}_k(t)$ be as in Definition \ref{Zsys}. For $k\in\mathcal K$,  define  the function 
$$W_k:(0,\infty)^{s_k+2}\times(0,\infty)^{k-1}\to\R$$ 
recursively as follows
\beq\label{Wk}
\left\{
\begin{aligned}
W_1(z)&=\mathcal R Z_1(z)&&\text{ for }k=1,\\ 
W_k(z,\zeta)&=\mathcal R Z_k(z,\zeta)+\sum_{j=1}^{k-1}W_j (z,\zeta) \frac{\partial Z_k(z,\zeta)}{\partial \zeta_j} &&\text{ for } k\ge 2.
\end{aligned}
\right.
\eeq

\begin{assumption}\label{B1} 
There exist  real numbers $\sigma\ge 0$, $\alpha\ge 1/2$,  
a strictly increasing, divergent sequence of positive numbers $(\mu_n)_{n=1}^\infty$ that preserves the addition and unit increment in the case $m_*=0$, or preserves the addition in the case $m_*\ge 1$, 
two increasing sequences $(M_n)_{n=1}^\infty$ and $(\widetilde M_n)_{n=1}^\infty$ of natural numbers as in Definition \ref{Hexpand}, 
and functions 
\beq \label{pncond}
p_n\in \widehat{\mathscr P}_{m_*}(M_n,\widetilde M_n,-\mu_n,G_{\alpha,\sigma,\C},G_{\alpha,\sigma})
\text{  for all $n\in\N$ }
\eeq 
such that $f(t)$ has the asymptotic expansion, in the sense of Definition \ref{Hexpand},
\beq\label{fseq}
f(t)\sim \sum_{n=1}^\infty p_n(\widehat\LL_{M_n}(t),\widehat{\mathcal Y}_{\widetilde M_n}(t)) \text{ in }G_{\alpha,\sigma}.
\eeq
\end{assumption}

We will construct the asymptotic expansion for a solution of \eqref{fctnse} using the following explicitly defined functions.
 
 \begin{definition}\label{construct}
Under Assumption \ref{B1}, define $q_n$, for $n\in\N$,  recursively as follows.
 In the case $m_*=0$, 
\beq\label{qn}
q_n=\mathcal Z_{A_\C}\Big(p_n - \sum_{\substack{1\le m,j\le n-1,\\ \mu_m+\mu_j=\mu_n}}B_\C(q_m,q_j) - \chi_n \Big) \quad\text{for } n \ge 1,
\eeq
with 
\beq \label{chin}
\chi_n=
\begin{cases}
\begin{displaystyle}
\mathcal R q_\lambda+\sum_{j=1}^{\widetilde M_\lambda} W_j\frac{\partial q_\lambda}{\partial \zeta_j} ,
\end{displaystyle}
& \text{if there is $\lambda\in [1, n-1]$ such that $\mu_\lambda+1=\mu_n$,}\\
 0,&\text{otherwise}.
\end{cases}
\eeq 
In the case $m_*\ge 1$, let $\chi_n=0$ for all $n$, and define $q_n$ by \eqref{qn}.
 \end{definition}

When $n=1$, \eqref{qn} and \eqref{chin} clearly mean 
$q_1=\mathcal Z_{A_\C} p_1$  and $\chi_1=0$.
For $n\in\N$, the index $\lambda$ in \eqref{chin}, if exists, is unique and $\lambda\le n-1$. Thus, equation \eqref{qn} is, indeed, a recursive formula in $n$. The main properties of $q_n$ are the following, which will be proved in Section \ref{proofsec}.

\begin{proposition}\label{qregpower}
For any $n\in\N$, one has 
\begin{align}
\label{qreg} 
q_n\in \widehat{\mathscr P}_{m_*}(M_n,\widetilde M_n,-\mu_n,G_{\alpha+1,\sigma,\C},G_{\alpha+1,\sigma}).
\end{align}
Consequently, one has, for any $n\in\N$,
\beq\label{Breg} 
B_\C( q_m,q_j) \in\widehat{\mathscr P}_{m_*}(M_n,\widetilde M_n,-\mu_n,G_{\alpha,\sigma,\C},G_{\alpha,\sigma})\text{ for  $m,j\in\N$ with  
$\mu_n=\mu_m+\mu_j$,}
\eeq
\beq\label{chireg}
\chi_n\in\widehat{\mathscr P}_{m_*}(M_n,\widetilde M_n,-\mu_n,G_{\alpha,\sigma,\C},G_{\alpha,\sigma}).
\eeq
\end{proposition}

Our first main result on the asymptotic expansion of the Leray--Hopf weak solutions is the next theorem.

\begin{theorem}\label{mainthm}
Under Assumption \ref{B1}, any Leray--Hopf weak solution $u(t)$  of \eqref{fctnse} has the asymptotic expansion, in the sense of Definition \ref{Hexpand},
 \beq\label{uexpand}
u(t)\sim  \sum_{n=1}^\infty q_n\left(\widehat\LL_{M_n}(t),\widehat{\mathcal Y}_{\widetilde M_n}(t)\right) \text{ in }G_{\alpha+1-\rho,\sigma}\text{ for any } \rho \in (0,1).
 \eeq
\end{theorem}

Thanks to property \eqref{qreg} in Propositoion \ref{qregpower}, the asymptotic expansion \eqref{uexpand} is at least well-defined in the sense of Definition \ref{Hexpand}. The proof of Theorem \ref{mainthm} is presented in Section \ref{proofsec}.

Observe in Theorem \ref{mainthm} above that both the force $f(t)$ and the solution $u(t)$ have  infinite series expansions. However, one can also consider the finite sum approximations, which means that there is $N_*\in\N$ such that  \eqref{remain} holds only for $N=N_*$.
If $f(t)$ has a finite sum approximation, then one can follow the proof in Section \ref{proofsec} below to obtain a corresponding finite sum approximation for the solution $u(t)$. For some calculations, see \cite[Theorem 4.1]{CaH1} and \cite[Theorem 5.6]{CaH2}.

\subsection{Statements with real spaces}\label{Rsec}

We reformulate  Theorems \ref{mainthm} using spaces $G_{\sigma,\alpha}$, but not $G_{\sigma,\alpha,\C}$.
For that purpose, we modify Definitions \ref{Hclass}, \ref{plusclass}, \ref{Zsys} and \ref{Hexpand}.

\begin{definition}\label{realPL}
Let $X$ be a real linear space.
Given integers $k\ge m\ge 0$.

\begin{enumerate}[label=\tnum]
 \item For $\ell\ge 1$, define the class $\widehat{\mathcal P}_m(k,\ell,X)$ to be the collection of functions which are the finite sums of the following functions
\beq\label{realpz}
(z,\zeta)\in (0,\infty)^{k+2}\times(0,\infty)^\ell\mapsto z^\beta \zeta^\gamma 
\left(\prod_{j=1}^{j_*} \sigma_j(\omega_j z_{\kappa_j})\right) 
\left(\prod_{j=1}^{\tilde j_*} \tilde\sigma_j(\tilde\omega_j \ln\zeta_{\tilde \kappa_j}) \right)  \xi,
\eeq
where $\xi\in X$, 
 $\beta\in \mathcal E_\R(m,k,0)$, $\gamma\in \R^\ell$, 
 $j_*,\tilde j_*\in\N$,
 $\kappa_j\in [0,k]$,  $\tilde \kappa_j\in [1,\ell]$, 
$\omega_j, \tilde\omega_j\in\R$,  
and, $\sigma_j,\tilde\sigma_j$ are functions in the set $\{\cos,\sin\}$.

\item Let  ${\mathcal P}_m(k,X)$,  be defined by the same way as $\widehat{\mathcal P}_m(k,\ell,X)$,  but for functions $p=p(z)$ only.
Denote $\widehat{\mathcal P}_m(k,0,X)=\mathcal P_m(k,X)$. 
\end{enumerate}
\end{definition}

Note that vector $\beta=(\beta_{-1},\beta_0,\ldots,\beta_k)$ in \eqref{realpz} satisfies
\beqs
\beta_{-1}=\beta_0=\ldots=\beta_m=0,\quad \beta_{m+1}, \ldots,\beta_k \in\R.
\eeqs

\begin{definition}\label{plusreal}
 Let $k\ge m\ge 0$. 
 
 \begin{enumerate}[label=\tnum]
\item \label{plri} For $\ell\ge 1$, define $\widehat{\mathcal P}_m^+(k,\ell)$ to a the set of functions in $p\in \widehat{\mathcal P}_m(k,\ell,\R)$  of the form 
\beq\label{prealde}
p(z,\zeta)=p_*(\zeta)+q(z,\zeta),
\eeq 
where 
\begin{enumerate}[label=\rnum]
\item 
$p_*(\zeta)$ is the same as in Definition \ref{plusclass}\ref{plusi}, and 

\item \label{plrb} $q\in \widehat{\mathcal P}_m(k,\ell,\R)$ can be written as a finite sum of the right-hand side of \eqref{realpz} with the same conditions on $\xi,\beta,\gamma,j_*,\tilde j_*,\omega_j,\tilde\omega_j,\kappa_j,\tilde\kappa_j,\sigma_j,\tilde\sigma_j$, and, additionally,
\beq \label{bneg}
\beta<\mathbf 0_{k+2}\text{ and ($\omega_j=0$ whenever $\kappa_j=0$).}
\eeq 
\end{enumerate}

\item Define ${\mathcal P}_m^+(k)$ in the same way as part \ref{plri}  but for functions $p=p(z)$ only. 
Specifically,
\beq\label{prc} 
p(z)=1+q(z),
\eeq 
where $q\in \mathcal P_m(k,\R)$ has the same properties as in the requirement \ref{plrb} of part \ref{plri} above without the presence of $\zeta^\gamma$ and the product $\prod_{j=0}^{j_*} \tilde\sigma_j(\tilde\omega_j \ln \zeta_{\tilde \kappa_j})$. 
Denote $\widehat{\mathcal P}_m^+(k,0)={\mathcal P}_m^+(k)$.
\end{enumerate}
\end{definition}

\begin{remark}\label{newlim}
Assuming $(X,\|\cdot\|_X)$ is a normed space, we denote by $Q(z,\zeta)$ the function in \eqref{realpz}. Then $\|Q(z,\zeta)\|_X\le z^\beta\zeta^\gamma \|\xi\|_X$.
If $\beta<\mathbf 0_{k+2}$  then $Q(\widehat \LL_k(t),\zeta)\to 0$, as $t\to\infty$, uniformly in $\zeta$ in compact subsets of $(0,\infty)^\ell$.
With this observation, we can assert that  the functions $q(z,\zeta)$ and $p(z,\zeta)$ in \eqref{prealde} as well as functions $q(z)$ and $p(z)$ in \eqref{prc} have the corresponding limits in \eqref{limq} and \eqref{limp}.
\end{remark}

\begin{definition}\label{Zreal}
Let $m\in\Z_+$. Define $\mathcal U(m)$ to be  the set of the triples $(\mathcal K,(s_k)_{k\in\mathcal K},(Z_k)_{k\in\mathcal K})$ such that the set $\mathcal K$ and numbers $s_k$ are the same as in Definition \ref{Zsys}, and the functions $Z_k$, for $k\in\mathcal K$, satisfy
$Z_k\in \widehat{\mathcal P}_m^+(s_k,k-1)$.
\end{definition}

For $(\mathcal K,(s_k)_{k\in\mathcal K},(Z_k)_{k\in\mathcal K})\in \mathcal U(m)$, one can see, thanks to Remark \ref{newlim}, that the limits in \eqref{Zlim} and \eqref{limYk} still hold true.

\begin{definition}\label{realex}
Let  $(X,\|\cdot\|_X)$ be a real normed space. Suppose $g$ is a function from $(T,\infty)$ to $X$ for some $T\in\R$. Let  $m_*\in \Z_+$, $(\mathcal K,(s_k)_{k\in\mathcal K},(Z_k)_{k\in\mathcal K})$ belong to $\mathcal U(m_*)$ and $\widehat{\mathcal Y}_k(t)$ be defined by \eqref{Y1}--\eqref{hatY}. 
Let $(\gamma_k)_{k=1}^\infty$, $(M_k)_{k=1}^\infty$, $(\widetilde M_k)_{k=1}^\infty$ be   sequences  as in Definition \ref{Hexpand}.

We say the function $g(t)$ has an asymptotic expansion
\beqs
g(t)\sim \sum_{k=1}^\infty p_k\left(\widehat{\LL}_{M_k}(t),\widehat{\mathcal Y}_{\widetilde M_k}(t)\right) L_{m_*}(t)^{-\gamma_k} , \text{ where $p_k\in \widehat{\mathcal P}_{m_*}(M_k,\widetilde M_k,X)$ for $k\in\N$, }
\eeqs
if, for each $N\in\N$, there is some number $\mu>\gamma_N$ such that
\beq\label{grrem}
\left\|g(t) - \sum_{k=1}^N p_k\left(\widehat{\LL}_{M_k}(t),\widehat{\mathcal Y}_{\widetilde M_k}(t)\right) L_{m_*}(t)^{-\gamma_k} \right\|_X=\bigo( L_{m_*}(t)^{-\mu}).
\eeq
\end{definition}

Note that the condition \eqref{grrem} is the same as \eqref{remain} except for requirements about the specific forms of the functions $Z_k$ and $p_k$. In fact, these forms are interchangeable  by the virtue of Lemma \ref{convert} below.
We now return to equation \eqref{fctnse}.

\begin{assumption}\label{Rhypo} 
There exist   $m_*\in \Z_+$, $(\mathcal K,(s_k)_{k\in\mathcal K},(Z_k)_{k\in\mathcal K})\in \mathcal U(m_*)$,
 and real numbers $\sigma\ge 0$, $\alpha\ge 1/2$,  
 sequences  $(\mu_n)_{n=1}^\infty$, $(M_n)_{n=1}^\infty$ and $(\widetilde M_n)_{n=1}^\infty$ as in Assumption \ref{B1},  
and functions 
\beq \label{pnreal}
\widehat p_n\in \widehat{\mathcal P}_{m_*}(M_n,\widetilde M_n,G_{\alpha,\sigma})\text{  for all $n\in\N$ }
\eeq 
such that $f(t)$ has the asymptotic expansion, in the sense of Definition \ref{realex},
\beq\label{freal}
f(t)\sim \sum_{n=1}^\infty \widehat p_n\left(\widehat{\LL}_{M_n}(t),\widehat{\mathcal Y}_{\widetilde M_n}(t)\right) L_{m_*}(t)^{-\mu_n} \text{ in }G_{\alpha,\sigma},
\eeq
where the functions $\widehat{\mathcal Y}_k(t)$ are defined by \eqref{Y1}--\eqref{hatY}. 
\end{assumption}

\begin{theorem}\label{thm3}
Under Assumption \ref{Rhypo}, there exist functions $\widehat q_n\in \widehat{\mathcal P}_{m_*}(M_n,\widetilde M_n,G_{\alpha+1,\sigma})$, for  $n\in\N$, such that any Leray--Hopf weak solution $u(t)$ of \eqref{fctnse}  admits the asymptotic expansion, in the sense of Definition \ref{realex},
\beq\label{ureal}
u(t)\sim \sum_{n=1}^\infty \widehat q_n\left(\widehat{\LL}_{M_n}(t),\widehat{\mathcal Y}_{\widetilde M_n}(t)\right) L_{m_*}(t)^{-\mu_n}  \text{ in $G_{\alpha+1-\rho,\sigma}$ for any $\rho\in(0,1)$.}
\eeq
\end{theorem}

The proof of Theorem \ref{thm3} will be given in Section \ref{proofsec}.

\subsection{Scenarios and examples}\label{example}

\begin{remark}\label{realrmk}
Consider the case in subsection \ref{Csec} when 
the functions $Z_k$, for $k\in\mathcal K$, belong to the class $\widehat{\mathscr P}_{m_*}(s_k,k-1,0,\R)$, corresponding to $\K=\R$ in Definition \ref{Hclass}, and, similarly, the functions $p_n$ in Assumption \ref{B1} belong to the class $\widehat{\mathscr P}_{m_*}(M_n,\widetilde M_n,-\mu_n,G_{\alpha,\sigma})$. Then there is  no need for the complexification and the proof of Theorem \ref{mainthm} is much simpler. All the functions $q_n$ in Definition \ref{construct} belong to  $\widehat{\mathscr P}_{m_*}(M_n,\widetilde M_n,-\mu_n,G_{\alpha+1,\sigma})$, the bilinear form $B_\C$ in \eqref{qn} is simply $B$, and, thanks to the last remark after Definition \ref{defZA}, the operator $\mathcal Z_{A_\C}$ in \eqref{qn} is simply $A^{-1}$.
\end{remark}

\begin{scenario}\label{scen1}
Assume $\alpha\ge 1/2$, $\sigma\ge 0$, and the following asymptotic expansion, as in Definition \ref{Hexpand},  for the force $f(t)$
\beq\label{fsam1}
f(t)\sim \sum_{n=1}^\infty p_n\left(\widehat{\LL}_{M_n}(t),\widehat{\mathcal Y}_{\widetilde M_n}(t)\right)
 \text{ in }G_{\alpha,\sigma}, 
\eeq
where $p_n\in \widehat{\mathscr P}_{m_*}(M_n,\widetilde M_n,-\gamma_n,G_{\alpha,\sigma,\C},G_{\alpha,\sigma})$.
Set
\begin{align} \label{genS1}
\mathcal S&=\left\{\sum_{j=1}^N m_j \gamma_j+\kappa:N\in\N,m_j\in \Z_+,\kappa\in \Z_+,\sum_{j=1}^\N m_j^2\ne 0
 \right\} 
 \text{ for }m_*=0,\\
\label{genS2}
\mathcal S&=\left\{\sum_{j=1}^N m_j \gamma_j:N\in\N,m_j\in \Z_+, \sum_{j=1}^\N m_j^2\ne 0\right\} 
\text{ for }m_*\ge 1.
\end{align} 
Then the set $\{\gamma_n:n\in\N\}$ is a subset of $\mathcal S$. Moreover, $\mathcal S$ preserves the addition and, in the case $m_*=0$, the unit increment.
We arrange $\mathcal S$ to be a sequence $(\mu_n)_{n=1}^\infty$ of strictly increasing numbers.
After inserting zero functions into the series in \eqref{fsam1} and reindexing $p_n$, $M_n$, $\widetilde M_n$, we obtain  from \eqref{fsam1} the asymptotic expansion \eqref{fseq} with \eqref{pncond}.
\end{scenario}

\begin{scenario}\label{scen2}
Assume in Definition \ref{Zsys} that $\mathcal K=\{1,2,\ldots,K^*\}$ and there is $N_*\in\N$ such that $s_k=N_*$ for $1\le k\le K^*$.
In particular, 
$Z_k\in \widehat{\mathscr P}_m(N_*,k-1)$ for $1\le k\le K^*$, $\mathcal Y_1(t)=Z_1(\widehat{\LL}_{N_*}(t))$, and $\mathcal Y_k(t)=Z_k(\widehat{\LL}_{N_*}(t),\mathcal Y_1(t),\ldots,\mathcal Y_{k-1}(t))$ for $2\le k\le K^*$.
Assume, additionally, that there is a fixed number $N\in\N$, and positive numbers $\gamma_n$
and functions $p_n\in \widehat {\mathscr P}_{m_*}(N_*,K^*,-\gamma_n,G_{\alpha,\sigma,\C},G_{\alpha,\sigma})$
 for $1\le n\le N$, such that 
\beq\label{fsam2}
\left|f(t)-\sum_{n=1}^N p_n\left(\widehat{\LL}_{N_*}(t),\widehat{\mathcal Y}_{K^*}(t)\right)\right|_{\alpha,\sigma}=\bigo(L_{m_*}(t)^{-\gamma})\text{ for all }\gamma>\gamma_K.
\eeq
Let $\mathcal S$ be defined as in \eqref{genS1} and \eqref{genS2} but with the number $N$ already being fixed above. Let $(\mu_n)_{n=1}^\infty$ be defined in the same way as in Scenario \ref{scen1}. Then, again, we obtain  from \eqref{fsam2} the asymptotic expansion \eqref{fseq}, or more specifically, 
\beqs
f(t)\sim \sum_{n=1}^\infty p_n\left(\widehat{\LL}_{N_*}(t),\widehat{\mathcal Y}_{K^*}(t)\right)
 \text{ in }G_{\alpha,\sigma}, 
\eeqs
where $p_n\in \widehat{\mathscr P}_{m_*}(N_*,K^*,-\mu_n,G_{\alpha,\sigma,\C},G_{\alpha,\sigma})$.
Note, in this case, that $p_n=0$ for sufficient large $n$.
Clearly, the scenario \eqref{fsam2} includes the case when we have the equation
\beqs
f(t)=\sum_{n=1}^{N} p_n\left(\widehat{\LL}_{N_*}(t),\widehat{\mathcal Y}_{K^*}(t)\right).
\eeqs
\end{scenario}

\begin{example}\label{eg1}
Let $\alpha\ge 1/2$ and $\sigma\ge 0$, and $\xi_1,\xi_2\in G_{\alpha,\sigma}$.
Assume, for large $t>0$,
\begin{align*}
f(t)&=\frac1{\ln t} \left \{  (\ln\ln t+1)^{1/3} + \left [ (L_3(t)+5L_4(t))^{1/2}+1\right]^{3/5}\right\}\xi_1\\
&\quad + \frac1{(\ln t)^{3}} \left\{( \ln\ln t+L_3(t)-1)^{1/4}+L_6(t)\right\}^{1/2}\xi_2.
\end{align*}
Let $m_*=1$.
We rewrite
\begin{align*}
f(t)&=\frac1{\ln t} \left\{  (\ln\ln t)^{1/3}\left (1+\frac1{\ln\ln t}\right)^{1/3} + L_3(t)^{3/10} \left[ \left(1+\frac{5L_4(t)}{L_3(t)} \right)^{1/2}+\frac1{L_3(t)^{1/2}}\right]^{3/5}\right\}\xi_1\\
&\quad + \frac{(\ln\ln t)^{1/8}}{(\ln t)^3}
\left \{ \left(1+\frac{L_3(t)-1}{\ln\ln t}\right)^{1/4}+\frac{L_6(t)}{(\ln\ln t)^{1/4}}\right\}^{1/2}\xi_2.
\end{align*}
Set 
\begin{align*}
\zeta_1&=Z_1(z_{-1},z_0,z_1,z_2)=1+z_2^{-1},\quad
\zeta_2=Z_2(z_{-1},\ldots,z_4)=1+5z_3^{-1}z_4,\\
\zeta_3&=Z_3(z_{-1},\ldots,z_4,\zeta_1,\zeta_2)=\zeta_2^{1/2}+z_3^{-1/2},\quad
\zeta_4=Z_4(z_{-1},\ldots,z_4,\zeta_1,\zeta_2,\zeta_3)=1+z_2^{-1}z_3-z_2^{-1},\\
\zeta_5&=Z_5(z_{-1},\ldots,z_6,\zeta_1,\ldots \zeta_4)=\zeta_4^{1/4}+z_2^{-1/4}z_6,
\end{align*}
and
\begin{align*}
p_1(z_{-1},\ldots,z_4,\zeta_1,\ldots,\zeta_3)&=z_1^{-1}\left (z_2^{1//2}\zeta_1^{1/3}+z_3^{3/4}\zeta_3^{3/5}\right)\xi_1,\\
p_2(z_{-1},\ldots,z_6,\zeta_1,\ldots,\zeta_5)&=z_1^{-3} z_2^{1/8} \zeta_5^{1/2}\xi_2.
\end{align*}
Then 
$f(t)= p_1(\widehat{\LL}_6(t),\widehat{\mathcal Y}_5(t))+p_2(\widehat{\LL}_6(t),\widehat{\mathcal Y}_5(t)).$
According to Scenario \ref{scen2} and Remark \ref{realrmk}, we can apply Theorem \ref{mainthm} to find an asymptotic expansion for any Leray--Hopf weak solution $u(t)$, namely,
$$u(t)\sim \sum_{n=1}^\infty \frac1{(\ln t)^n} \widehat q_n(\widehat{\LL}_6(t),\widehat{\mathcal Y}_5(t) )
\text{ in $G_{\alpha+s,\sigma}$ for any $s\in[0,1)$,}
$$
where $\widehat q_n= \widehat q_n(z_{-1},\ldots,z_6,\zeta_1,\ldots,\zeta_5) \in \widehat{\mathscr P}_1(6,5,0,G_{\alpha+1,\sigma})$.
 \end{example}

\begin{example}\label{eg2}
Let $\xi\in V$ and, for large $t$,
$$ f(t)=\frac1{t^{1/3}}\sin (2t) (\ln t)^{2/3} \cos(\ln\ln t)\left (2+\frac{\cos(4\ln(1+3/\ln t))}{\ln\ln t}\right)^{1/2} \xi.$$
First, let $m_*=0$ and rewrite
$$ f(t)=\frac{\sqrt 2}{t^{1/3}}\sin(2 t) (\ln t)^{2/3}\cos(\ln\ln t) \left (1+\frac{\cos(4\ln(1+3/\ln t))}{2\ln\ln t}\right)^{1/2} \xi.$$
Set 
\begin{align*}
&\zeta_1=Z_1(z_{-1},z_0,z_1)=1+3z_1^{-1},\quad
\zeta_2=Z_2(z_1.z_2,\zeta_1)=1+\frac{\cos(4\ln \zeta_1)}{2z_2},\\
&\widehat p_1(z_{-1},\ldots,z_2,\zeta_1,\zeta_2)=\sqrt 2 z_1^{2/3} \zeta_2^{1/2}\sin(2 z_0)(\cos z_2) \xi\text{ and }\gamma_1=1/3.
\end{align*}
Then we have
$f(t)= t^{-\gamma_1}\widehat p_1(\widehat{\LL}_2(t),\widehat{\mathcal Y}_2(t)).$
Using the same arguments as for Scenario \ref{scen2}, we can apply Theorem \ref{thm3} to $\alpha=1/2$, $\sigma=0$.
It results in the asymptotic expansion, in the sense of Definition \ref{realex}, 
$$u(t)\sim \sum_{n=1}^\infty \frac1{t^{n/3}} \widehat q_n(\widehat{\LL}_2(t),\widehat{\mathcal Y}_2(t) )
\text{ in $\mathcal D(A^s)$ for any $s\in[0,3/2)$,}$$
for any Leray--Hopf weak solution $u(t)$, 
 where the functions $$\widehat q_n= \widehat q_n(z_{-1},z_0,z_1,z_2,\zeta_1,\zeta_2) \in \widehat{\mathcal P}_1(2,2,0,\mathcal D(A^{3/2}))
 \text{ are independent of $u(t)$.}$$
\end{example}

\section{Preparations for the proofs}\label{prepsec}

\subsection{Properties of spaces and operators}\label{mainprops}

We establish some properties for functions and operators involved in the statements of Theorems \ref{mainthm} and \ref{thm3}. 

\begin{lemma}\label{prodreal}
Consider either 
\begin{enumerate}[label=\rnum]
\item \label{pa}   $\K=\C$ or $\R$, $X$ is a linear space over $\K$, $\tilde X=X$ and $\tilde \K=\K$, or

\item\label{pb}  $X$ is a real linear space, $X_\C$ is its complexification, $\tilde X=(X_\C,X)$ and $\tilde \K=(\C,\R)$.
\end{enumerate}

 If $p\in \widehat{\mathscr P}_m(k,\ell,\mu,\tilde X)$ and $q\in \widehat{\mathscr P}_m(k,\ell,\mu',\tilde \K)$, then 
 the product $qp$ belongs to $\widehat{\mathscr P}_m(k,\ell,\mu+\mu',\tilde X) $.
\end{lemma}

\begin{proof}
  Case \ref{pa} can be proved by simple calculations using the form \eqref{pzedef} for $p$ and $q$. We consider Case \ref{pb}.
 Thanks to \eqref{symsum}, it suffices to assume 
 \beq\label{pqform}
 p(z,\zeta)=z^\beta\zeta^\gamma \xi + z^{\overline \beta}\zeta^{\overline \gamma}\,\overline \xi \text{ and }
 q(z,\zeta)=z^{\beta'}\zeta^{\gamma'} c + z^{\overline {\beta'}}\zeta^{\overline {\gamma'}}\overline{c}
 \eeq 
 for some $\xi\in X_\C$ and $c\in\C$.
 We have $ (qp)(z,\zeta)= Q_1(z,\zeta)+ Q_2(z,\zeta)$, where
 \begin{align*}
 Q_1(z,\zeta)&=
  c z^{\beta+\beta'}\zeta^{\gamma+\gamma'} \xi + \overline c z^{\overline\beta+\overline{\beta'}}\zeta^{\overline\gamma+\overline{\gamma'}} \overline\xi,\quad 
Q_2(z,\zeta) =\overline c z^{\beta+\overline{\beta'}}\zeta^{\gamma+\overline{\gamma'}} {\xi} +  c z^{\overline\beta+\beta'}\zeta^{\overline\gamma+\gamma'}\overline \xi.
 \end{align*}
It is clear from these formulas that $Q_1$ and $Q_2$  belong to $\widehat{\mathscr P}_m(k,\ell,\mu+\mu',X_\C)$, and satisfy the comjugation condition \eqref{xiconj}.
Therefore, $Q_1,Q_2$ belong to $\widehat{\mathscr P}_m(k,\ell,\mu+\mu',X_\C,X) $, and, hence, so does their sum $qp$.
\end{proof}

\begin{lemma}\label{derivze}
Consider either 
\begin{enumerate}[label=\rnum]
\item \label{da}  $\K=\C$ or $\R$, $X$ is a normed space over $\K$, $\tilde X=X$ and $\tilde \K=\K$, or

\item \label{db} $X$ is a real inner product space, $X_\C$ is its complexification, $\tilde X=(X_\C,X)$ and $\tilde \K=(\C,\R)$.
\end{enumerate}

If $p\in\widehat{\mathscr P}_{m}(k,\ell,\mu,\tilde X)$ and $1\le j\le \ell$, then $\partial p/\partial \zeta_j\in\widehat{\mathscr P}_{m}(k,\ell,\mu,\tilde X)$.
\end{lemma}
\begin{proof}
 Again, Case \ref{da} can be proved by simple calculations using the form \eqref{pzedef} for $p$ and $q$. Considering Case \ref{db}, without loss of the generality, we assume $p(z,\zeta)$ as in \eqref{pqform}. We calculate
$$ \frac{\partial p}{\partial \zeta_j}(z,\zeta)= \gamma_j \zeta_j^{-1}z^\beta\zeta^\gamma \xi + \overline{\gamma_j} \zeta_j^{-1}z^{\overline \beta}\zeta^{\overline \gamma}\,\overline \xi,  $$
and obtain the desired statement.
\end{proof}

We observe from Definitions \ref{defMR} and \ref{defZA} that 
\beq\label{ZAM}
 (A_\C+\mathcal M_{-1})(\mathcal Z_{A_\C}p)=p\quad \forall p\in\widehat{\mathscr P}_{-1}(k,0,H_\C),
 \eeq
\beq\label{ZAM2}
 \mathcal Z_{A_\C}((A_\C+\mathcal M_{-1})p)=p\quad \forall p\in \widehat{\mathscr P}_{-1}(k,0,G_{1,0,\C}).
 \eeq
In \eqref{ZAM} and \eqref{ZAM2} above, $\mathcal M_{-1}$ is understood to be defined for $\K=\C$ and $X=G_{1,0,\C}$.
More relations are obtained next.

\begin{lemma}\label{invar2}
Let $\K$, $X$ and $\tilde X$ be as in \ref{pa} or \ref{pb} of  Lemma \ref{prodreal}.
Given numbers $m,\ell\in \Z_+$ and $\mu\in\R$.
The following statements hold true.
\begin{enumerate}[label=\tnum]
\item\label{invi}  
	If $k\ge j\ge -1$, then $\mathcal M_j$ maps $\widehat{\mathscr P}(k,\ell,\tilde X)$ into itself.\\ 
	 If  $k\ge 0$, then
	  $\mathcal R$ maps $\widehat{\mathscr P}(k,\ell,\tilde X)$ into itself.
\item\label{invii} 	
	If $k\ge m\ge 0$ and $k\ge j\ge -1$, then $\mathcal M_j$ maps $\widehat{\mathscr P}_m(k,\ell,\mu,\tilde X)$ into itself.
\item \label{inviii} 
If $k\ge 0$, then $\mathcal R$ 
maps  $\widehat{\mathscr P}_{0}(k,\ell,\mu,\tilde X)$ into $\widehat{\mathscr P}_{0}(k,\ell,\mu-1,\tilde X)$. 
\end{enumerate}
\end{lemma}
\begin{proof}
Since \ref{invi}--\ref{inviii} can be verified easily in Case \ref{pa}, we consider Case \ref{pb} now.
The case $\ell=0$ is \cite[Lemma 6.4]{H6}.
Consider the case $\ell\ge 1$. We use the identifications in \eqref{samespace} and \eqref{same2}, with the operator $\mathcal R$ still only applying to the $z$-part of the variables $\tilde z$ in \eqref{zbtil}. Then the statements are reduced to the case $\ell=0$. The proofs in  \cite[Lemma 6.4]{H6} for the variable $\tilde z$  in this case $\ell=0$ still work and we obtain the desired statements.
\end{proof}

\begin{lemma}\label{invar4}
Let $\alpha,\sigma\ge 0$, $m,\ell\in\Z_+$ and $\mu\in\R$. Consider either
\begin{enumerate}[label=\rnum]
\item \label{Xa} $\K=\C$, $X_\alpha=G_{\alpha,\sigma,\C}$ and $X_{\alpha+1}=G_{\alpha+1,\sigma,\C}$, or

\item \label{Xb}   $X_\alpha=(G_{\alpha,\sigma,\C},G_{\alpha,\sigma})$ and $X_{\alpha+1}=(G_{\alpha+1,\sigma,\C},G_{\alpha+1,\sigma})$.
\end{enumerate}
One has the following.
\begin{enumerate}[label=\tnum]
\item\label{ZAi}	 
If $k\ge -1$,  then $\mathcal Z_{A_\C}$ maps $\widehat{\mathscr P}_{-1}(k,\ell,0,X_\alpha)$  into $\widehat{\mathscr P}_{-1}(k,\ell,0,X_{\alpha+1})$. 

\item\label{ZAii}	 If $k\ge m$, then $\mathcal Z_{A_\C}$ maps  $\widehat{\mathscr P}_{m}(k,\ell,\mu,X_\alpha)$ into $\widehat{\mathscr P}_{m}(k,\ell,\mu,X_{\alpha+1})$. 
\end{enumerate}
\end{lemma}
\begin{proof}
The statements in Case \ref{Xa} are direct consequences of Lemma \ref{Aioinv}.
We consider Case \ref{Xb}.
 For $\ell=0$, part \ref{ZAi} can be proved in the same way as in Lemmas 4.6(i) and 6.5(i) of \cite{H6},
 while part \ref{ZAii} is Lemma 6.5(ii) of \cite{H6}. (Their proofs mainly make use of Lemma \ref{Aioinv} and property \eqref{Acinv}.)
 The case $\ell\ge 1$ is obtained by using the identifications in \eqref{samespace} and \eqref{same2}  to reduce to the case $\ell=0$.
\end{proof}

In particular, when  $k\ge m=-1$, $\alpha=\sigma=\mu=0$, one has from Lemma \ref{invar4}\ref{ZAii} that $\mathcal Z_{A_\C}$ maps $\widehat{\mathscr P}_{-1}(k,\ell,0,H_\C,H)$  into $\widehat{\mathscr P}_{-1}(k,\ell,0,G_{1,0,\C},G_{1,0})$.

\begin{assumption}\label{Asys}
From here to the end of this section \ref{prepsec}, assume  
$$m\in \Z_+\text{ and }
(\mathcal K,(s_k)_{k\in\mathcal K},(Z_k)_{k\in\mathcal K})\in \mathscr U(m).$$
Let ${\mathcal Y}_k(t)$, $\widehat{\mathcal Y}_k(t)$ be defined as in Definition \ref{Zsys}, and $W_k$ be defined by \eqref{Wk}.
\end{assumption}

Let $(X,\|\cdot\|_X)$ be a normed space over $\K=\R$ or $\C$.
For any $p\in \mathscr P(k,X)=\widehat{ \mathscr P}(k,0,X)$, we have, thanks to the chain rule and \eqref{Lmderiv},
\beq\label{dpL}
\ddt p(\widehat{\LL}_k(t))
=\sum_{j=-1}^k  \ddt L_j(t)  \frac{\partial p}{\partial z_j}
=(\mathcal M_{-1}p+\mathcal Rp)\circ \widehat{\LL}_k(t).
\eeq
Similarly, for any $p\in \widehat{\mathscr P}(k,\ell,X)$ with $\ell\in\mathcal K$, we have
\begin{align*}
\ddt p(\widehat{\LL}_k(t),\widehat Y_\ell(t))
&=\sum_{j=-1}^k \ddt L_j(t) \frac{\partial p}{\partial z_j} (\widehat{\LL}_k(t),\widehat Y_\ell(t)) 
+\sum_{j=1}^\ell \ddt \mathcal Y_j(t)\frac{\partial p}{\partial \zeta_j} (\widehat{\LL}_k(t),\widehat Y_\ell(t)),
\end{align*}
which yields
\beq\label{dtpLY}
\ddt p(\widehat{\LL}_k(t),\widehat Y_\ell(t))
=(\mathcal M_{-1}p+\mathcal Rp)\circ (\widehat{\LL}_k(t),\widehat Y_\ell(t))
+\sum_{j=1}^\ell\ddt \mathcal Y_j(t) \frac{\partial p}{\partial \zeta_j} (\widehat{\LL}_k(t),\widehat Y_\ell(t)).
\eeq
Clearly, \eqref{dpL} is a particular case of \eqref{dtpLY} with $\ell=0$.
For $k\in\mathcal K$, we apply formula \eqref{dtpLY} to function $p:=Z_k$, 
numbers $k:=s_k$, $\ell:=k-1$
and space $X=\C$ noticing that, thanks to the second condition in \eqref{bposcond}, $\mathcal M_{-1}Z_k=0$.
We obtain the following recursive formula
\beq\label{dYk}
\ddt \mathcal Y_k(t)=(\mathcal RZ_k)\circ (\widehat{\LL}_{s_k}(t),\widehat{\mathcal Y}_{k-1}(t))
+\sum_{j=1}^{k-1}\ddt \mathcal Y_j(t) \frac{\partial Z_k}{\partial \zeta_j} (\widehat{\LL}_{s_k}(t),\widehat{\mathcal Y}_{k-1}(t)).
\eeq
We claim that
\beq\label{dYW}
\ddt \mathcal Y_k(t)
=W_k(\widehat{\LL}_{s_k}(t),\widehat{\mathcal Y}_{k-1}(t)).
\eeq
Indeed, it follows \eqref{dYk} that
\begin{align*}
\ddt \mathcal Y_1(t)&=(\mathcal R Z_1)(\widehat{\LL}_{s_1}(t))=W_1(\widehat{\LL}_{s_1}(t)),\\
\ddt \mathcal Y_2(t)&=\left.\left(\mathcal RZ_2+W_1 \frac{\partial Z_1}{\partial \zeta_1} \right)\right|_{(\widehat{\LL}_{s_2}(t),{\mathcal Y}_1(t))} 
=W_2(\widehat{\LL}_{s_2}(t),{\mathcal Y}_1(t)),
\end{align*}
and, recursively,
\beqs
\ddt \mathcal Y_k(t)=
\left.\left(\mathcal R Z_k+\sum_{j=1}^{k-1} W_j\frac{\partial Z_k}{\partial \zeta_j}\right)\right|_{(\widehat{\LL}_{s_k}(t),\widehat{\mathcal Y}_{k-1}(t))}
=W_k(\widehat{\LL}_{s_k}(t),\widehat{\mathcal Y}_{k-1}(t)).
\eeqs
Therefore, \eqref{dYW} is true.
Thanks to identity \eqref{dYW}, formula \eqref{dtpLY} is refined to
\beq\label{dpYL2}
\ddt p(\widehat{\LL}_k(t),\widehat{\mathcal Y}_\ell(t))
=\left.\left( \mathcal M_{-1}p+\mathcal R p+\sum_{j=1}^\ell W_j\frac{\partial p}{\partial \zeta_j} \right)\right|_{(\widehat{\LL}_{s_k}(t),\widehat{\mathcal Y}_\ell(t))}.
\eeq 

Below we determine the classes for the functions $W_k$ and the others on the right-hand side of \eqref{dpYL2}.

\begin{lemma}\label{Wlem}
For all $k\in\mathcal K$,
\beq \label{Wprop}
W_k\in\widehat{\mathscr P}_0(s_k,k-1,-1,\C,\R).
\eeq 
\end{lemma}
\begin{proof}
We prove \eqref{Wprop} by induction.
First of all, observe that
$$Z_1\in \widehat{\mathscr P}_m^+(s_1,0)\subset \widehat{\mathscr P}_0(s_1,0,0,\C,\R).$$
Then by Lemma \ref{invar2}\ref{inviii}, we have
$$W_1=\mathcal R Z_1\in  \widehat{\mathscr P}_0(s_1,0,-1,\C,\R),$$
 hence \eqref{Wprop} is true for $k=1$.
Let $k\in\mathcal K$, $k\ge 2$ and assume
\beq \label{ihypo}
W_j\in\widehat{\mathscr P}_0(s_j,j-1,-1,\C,\R)
 \text{ for all }1\le j\le k-1.
 \eeq 
 By \eqref{Zs} and \eqref{PPrel}, 
 $$Z_k\in \widehat{\mathscr P}_m^+(s_k,k-1)\subset \widehat{\mathscr P}_m(s_k,k-1,0,\C,\R)\subset \widehat{\mathscr P}_0(s_k,k-1,0,\C,\R).$$
 Let $1\le j\le k-1$.
Thanks to Lemma \ref{invar2}\ref{inviii} and Lemma \ref{derivze}, we have 
\beq \label{fact}
\mathcal R Z_k \in \widehat{\mathscr P}_0(s_k,k-1,-1,\C,\R)
 \text{ and }\frac{\partial Z_k}{\partial \zeta_j} \in \widehat{\mathscr P}_0(s_k,k-1,0,\C,\R).
\eeq 
 By the induction hypothesis \eqref{ihypo},  
 $$W_j\in\widehat{\mathscr P}_0(s_j,j-1,-1,\C,\R)\subset \widehat{\mathscr P}_0(s_k,k-1,-1,\C,\R).$$
Then thanks to Lemma \ref{prodreal},  the product 
$$W_j\frac{\partial Z_k}{\partial \zeta_j} \text{  belongs to }\widehat{\mathscr P}_0(s_k,k-1,-1,\C,\R).$$
 This fact and the first property in \eqref{fact} imply
 \beqs
W_k=\mathcal R Z_k+\sum_{j=1}^{k-1} W_j 
\frac{\partial Z_k}{\partial \zeta_j} \in \widehat{\mathscr P}_0(s_k,k-1,-1,\C,\R),
\eeqs
which proves that \eqref{Wprop} is true for the considered $k$.
By the Induction Principle, \eqref{Wprop} is true for all $k\in \mathcal K$. 
 \end{proof}

 \begin{lemma}\label{Wplem}
Let $\K$, $X$ and $\tilde X$ be as in  Lemma \ref{derivze}. Suppose $\ell\in\mathcal K$,  $k\ge s_\ell$ and $\mu\in\R$.  

If $p\in \widehat{\mathscr P}_0(k,\ell,\mu,\tilde X)$, then
\beq\label{pest4}
\mathcal R p+\sum_{j=1}^\ell W_j \frac{\partial p}{\partial \zeta_j} \text{ belongs to } \widehat{\mathscr P}_0(k,\ell,\mu-1,\tilde X).
\eeq

If $k\ge m\ge 1$   and $p\in \widehat{\mathscr P}_m(k,\ell,\mu,\tilde X)$, then
\beq\label{pest8}
\mathcal R p+\sum_{j=1}^\ell W_j \frac{\partial p}{\partial \zeta_j} \text{ belongs to } \widehat{\mathscr P}_0(k,\ell,-1,\tilde X).
\eeq
\end{lemma}
\begin{proof}
Assume $p\in \widehat{\mathscr P}_0(k,\ell,\mu,\tilde X)$.
One has, by Lemma \ref{invar2}\ref{inviii}, 
\beq \label{RpX}
\mathcal R p\in \widehat{\mathscr P}_0(k,\ell,\mu-1,\tilde X),
\eeq 
and, by Lemma \ref{derivze}, $\partial p/\partial \zeta_j \in \widehat{\mathscr P}_0(k,\ell,\mu,\tilde X)$ for $1\le j\le \ell$. Combining the last property with  \eqref{Wprop} applied to  $k:=j$   and Lemma \ref{prodreal}, and also noticing that $s_j\le k$, we find  $$W_j\partial p/\partial \zeta_j \in \widehat{\mathscr P}_0(k,\ell,\mu-1,\tilde X).$$
With this fact, we can sum up $W_j\partial p/\partial \zeta_j$ in $j$ and use \eqref{RpX},
to obtain \eqref{pest4}.

In the case  $p\in \widehat{\mathscr P}_m(k,\ell,\mu,\tilde X)$ with $m\ge 1$, we have, from \eqref{Pmmz} with $m'=0<m$, that  $p\in \widehat{\mathscr P}_0(k,\ell,0,\tilde X)$. Then we can apply \eqref{pest4} to $\mu=0$ to obtain \eqref{pest8}.
\end{proof}

\subsection{An asymptotic approximation result}\label{linear}
In this subsection, we establish an asymptotic approximation for solutions of the linearized NSE with a decaying force.
It is a counter part of \cite[Theorem 5.1]{H6} but for functions with subordinate variables. 
Recall that we are under Assumption \ref{Asys} throughout this subsection.

\begin{theorem}\label{Fode}
Given numbers $\alpha,\sigma\ge 0$, $\mu>0$,  integers $k\ge m$ and $\ell\ge 0$ with $k\ge s_\ell$ in the case $\ell\ge 1$. 
Let $p$ be in $\widehat{\mathscr P}_m(k,\ell,-\mu,G_{\alpha,\sigma,\C})$  and satisfy 
\beqs
p(z,\zeta)\in G_{\alpha,\sigma} \text{ for all }z\in (0,\infty)^{k+2},\zeta\in(0,\infty)^\ell.
\eeqs
Let $T_*>E_{k}(0)$ such that ${\mathcal Y}_j(t)\ge 1/2$ for all $1\le j\le \ell$ and $t\ge T_*$.
Suppose  $g$ is a function from $[T_*,\infty)$ to $G_{\alpha,\sigma}$ that satisfies 
  \beqs
  |g(t)|_{\alpha,\sigma}\le M  L_m(t)^{-\mu-\delta_0} \text{ a.e. in $(T_*,\infty)$,}
  \eeqs 
for some positive numbers $\delta_0$ and $M$. 

Suppose  $w\in C([T_*,\infty),H_{\rm w})\cap L^1_{\rm loc}([T_*,\infty),V) $, with $w'\in L^1_{\rm loc}([T_*,\infty),V')$,  is a weak solution of 
 \beq\label{weq}
 w'=-Aw+p\left(\widehat\LL_k(t),\widehat{\mathcal Y}_\ell(t)\right)+g(t) \text{ in $V'$ on $(T_*,\infty)$,} 
 \eeq
    i.e., it holds, for all $v\in V$, that
  \beqs
 \ddt \inprod{w,v}=-\doubleinprod{w,v}+\inprod{p\left(\widehat\LL_k(t),\widehat{\mathcal Y}_\ell(t)\right)+g(t),v} \text{ in the distribution sense on $(T_*,\infty)$.}
  \eeqs
  
  Assume, in addition, $w(T_*)\in G_{\alpha,\sigma}$.   
  Then the following statements hold true.
\begin{enumerate}[label=\rm (\roman*)]
 \item\label{api} $w(t)\in G_{\alpha+1-\varep,\sigma}$ for all $\varep\in(0,1)$ and $t>T_*$.
 
 \item\label{apii} 
 Let 
 \beqs
 \delta_*=\begin{cases}
 \text{any number in $(0,1)\cap (0,\delta_0]$},&\text{ for } m=0,\\
 \delta_0,& \text{ for } m\ge 1.
 \end{cases}
 \eeqs
For any  $\varep\in(0,1)$, there exists a  constant $C>0$  such that 
\beq\label{wremain}
\left |w(t)-(\mathcal Z_{A_\C} p)\left(\widehat\LL_k(t),\widehat{\mathcal Y}_\ell(t)\right)\right |_{\alpha+1-\varep,\sigma}\le C L_m(t)^{-\mu-\delta_*}  \text{ for all }t \ge T_*+1. 
\eeq
\end{enumerate}
\end{theorem}

We prepare for the proof of Theorem \ref{Fode} with the following calculations and estimates.

Consider   $k,\ell,\mu,T_*$ as in Theorem \ref{Fode}, power vectors 
$\beta_*=(0,\beta_0,\ldots,\beta_k)\in\mathcal E_\C(m,k,-\mu)$ and $\gamma\in \R^\ell$. 
Let $p:(0,\infty)^{k+2}\times(0,\infty)^\ell\to \R$ be defined by $p(z,\zeta)=z^{\beta_*}\zeta^\gamma$.
Define a function $h:[0,\infty)\to (0,\infty)$ by  
$$h(t)=p(\widehat\LL_k(T_*+t),\widehat{\mathcal Y}_\ell(T_*+t))
= \widehat\LL_k(T_*+t)^{\beta_*} \widehat{\mathcal Y}_\ell(T_*+t)^{\gamma} \text{ for $t\ge 0$. }$$
By  \eqref{LLo}, one has
\beq \label{he}
|h(t)| =\bigo( L_m(T_*+t)^{-\mu+s'})\text{ for all } s'>0.
\eeq 
Note that $\mathcal M_{-1}p=0$. The derivative of $h(t)$ can be computed, thanks to \eqref{dpYL2}, by
 \beq\label{dh}
 h'(t)
 =\left.\left(\mathcal R p+\sum_{j=1}^\ell W_j \frac{\partial p}{\partial \zeta_j} \right)\right|_{(\widehat{\LL}_{s_k}(T_*+t),\widehat{\mathcal Y}_\ell(T_*+t))}.
 \eeq

Let $s$ be an arbitrary number in $(0,1)$.
 
Consider $m=0$. It follows \eqref{pest4} that 
\beq\label{Rh1}
\mathcal R p+\sum_{j=1}^\ell W_j \frac{\partial p}{\partial \zeta_j} \text{ belongs to } \widehat{\mathscr P}_0(k,\ell,-\mu-1,\C).
\eeq
 By  \eqref{dh} and \eqref{Rh1}, we have, same as \eqref{he},
$ |h'(t)|=\bigo( (T_*+t)^{-\mu-s})$.
Combining this estimate with the continuity of the functions $h'(t)$ and $ (T_*+t)^{-\mu-s}$ on $[0,\infty)$, we obtain
\beq\label{dF1}
|h'(t)|\le C (T_*+t)^{-\mu-s} \text{ for all }t\ge 0, \text{ with some constant $C>0$ depending on $s$.}
 \eeq

Consider $m\ge 1$. Thanks to \eqref{pest8}, one has
\beq\label{Rh2}
\mathcal R p+\sum_{j=1}^\ell W_j \frac{\partial p}{\partial \zeta_j} \text{ belongs to } \widehat{\mathscr P}_0(k,\ell,-1,X).
\eeq
Combining \eqref{dh} with \eqref{Rh2} yields
$ |h'(t)|=\bigo\left( (T_*+t)^{-s}\right)$.
We obtain, same as for \eqref{dF1}, that 
 \beq\label{dF2}
|h'(t)|\le C (T_*+t)^{-s} \text{ for all }t\ge 0,\text{ with some constant $C>0$ depending on $s$.}
 \eeq

We are now ready to prove Theorem \ref{Fode}.

 \begin{proof}[Proof of Theorem \ref{Fode}]
Because the case $\ell=0$ is \cite[Theorem 5.1]{H6}, we only consider $\ell\ge 1$ below.
Part \ref{api} is from Lemma 2.4(i) of \cite{CaH1}. 
For part \ref{apii}, we follow the proof of \cite[Theorem 5.1]{H6} closely.
By \eqref{Linc} and \eqref{Lone}, the function $t\in[0,\infty)\mapsto  L_m(T_*+t)$ is increasing with images belonging to $[1,\infty)$.
 Denote $\Xi(t)=(\widehat\LL_k(t),\widehat{\mathcal Y}_\ell(t))$.
 
Let $N\in\N$, denote $\Lambda=\Lambda_N$, let $A_{\Lambda}=A|_{P_{\Lambda} H}$ and $A_{\C,\Lambda}=A_\C|_{P_{\Lambda,\C} H_\C}=(A_{\Lambda})_\C$.
 Then $A_{\Lambda}$, respectively, $A_{\C,\Lambda}$, is an invertible linear mapping from $P_{\Lambda} H$, respectively,  $P_{\Lambda,\C} H_\C$, to itself.
By applying projection $P_{\Lambda}$ to equation \eqref{weq} and using the variation of constants formula, we have, for $t\ge 0$, 
\beq\label{vcf}
\begin{aligned} 
P_{\Lambda} w(T_*+t)&=e^{-tA_{\Lambda}}P_{\Lambda} w(T_*) +\int_0^t e^{-(t-\tau)A_{\Lambda}}P_{\Lambda} p(\Xi(T_*+\tau))\d\tau\\
&\quad +\int_0^t e^{-(t-\tau)A_{\Lambda}}P_{\Lambda} g(T_*+\tau)\d\tau.
 \end{aligned}
 \eeq 

 Assume  
\beqs
p(z,\zeta)=\sum_{(\beta,\gamma)\in S} p_{\beta,\gamma}(z),
\text{ where } p_{\beta,\gamma}(z)=z^{ \beta} \zeta^\gamma\xi_{\beta,\gamma}\text{ with } \xi_{\beta,\gamma}\in G_{\alpha,\sigma,\C},
\eeqs
and $S\subset \C^{k+2}\times \C^\ell$ is the same as in Definition \ref{Hclass}\ref{Hii}.
For $(\beta,\gamma)\in S$, by letting $\beta=(\beta_{-1},\beta_0,\ldots,\beta_k)\in S$, we denote $\widetilde\beta=(0,\beta_0,\ldots,\beta_k)$ and 
  \beqs 
  h_{\beta,\gamma}(t)=\widehat\LL(T_*+t)^{\widetilde\beta}\widehat{\mathcal Y}_\ell(T_*+t)^\gamma.
   \eeqs  
   Then we can write
   \beq\label{ph}
   p_{\beta,\gamma}(\Xi(T_*+t)) =e^{\beta_{-1} (T_*+t)}h_{\beta,\gamma}(t) \xi_{\beta,\gamma}.
\eeq
Define 
$$\mathcal{A}_{\beta,\Lambda}=(A_\C+\beta_{-1}{\mathbb I})|_{P_{\Lambda,\C}H_\C}=A_{\C,\Lambda}+\beta_{-1} {\rm Id}_{P_{\Lambda,\C}H_\C}.$$
Then $\mathcal{A}_{\beta,\Lambda}$ is an invertible linear mapping from $P_{\Lambda,\C} H_\C$ to itself.
Define
\beq\label{Jbeta}
J_{\beta,\gamma}(t)=\int_0^t h_{\beta,\gamma}'(\tau)  \mathcal{A}_{\beta,\Lambda}^{-1} e^{-(t-\tau)\mathcal{A}_{\beta,\Lambda}}P_{\Lambda,\C}  \xi_{\beta,\gamma} \d\tau.
\eeq
 The above $J_{\beta,\gamma}(t)$ will appear when we compute the second term on the right-hand side of \eqref{vcf} by using formula \eqref{ph} and integration by parts.

 Let $\varep\in(0,1)$. We obtain the same  estimate (5.16) in \cite{H6} which states that
\beq \label{ynorm}
  \begin{aligned}
 & \left |P_{\Lambda,\C}\Big(w(T_*+t)-(\mathcal Z_{A_\C}p)(\Xi(T_*+t))\Big)\right |_{\alpha+1-\varep,\sigma}
  \le |e^{-tA_\C}P_{\Lambda,\C} W_*|_{\alpha+1,\sigma}\\
    &\quad  +\sum_{(\beta,\gamma)\in S} | J_{\beta,\gamma}(t)|_{\alpha+1,\sigma}+\left|\int_0^t e^{-(t-\tau)A}P_{\Lambda} g(T_*+\tau)\d\tau\right|_{\alpha+1-\varep,\sigma},
    \end{aligned}
\eeq  
where  $W_*=w(T_*) - (\mathcal Z_{A_\C}p)(\widehat\LL_k(T_*))$. We have $W_*\in G_{\alpha,\sigma,\C}$ and $| W_*|_{\alpha,\sigma}\in[0,\infty)$.

 Let $\delta,\theta\in(0,1)$ and $t\ge 1$.  
 For the rest of this proof,  $C$ denotes a generic positive constant which is independent of $N$ and $t$, and may assume different values from place to place.
 
For the first term on the right-hand side of \eqref{ynorm}, estimate (5.17) in \cite{H6} gives
  \beq\label{term1}
     |e^{-tA_\C}P_{\Lambda,\C} W_*|_{\alpha+1,\sigma}
  \le \frac{e^{-\delta t}}{e(1-\delta)} | W_*|_{\alpha,\sigma}\le C L_m(T_*+t)^{-\mu-\delta_*}.
\eeq

For the second term on the right-hand side of \eqref{ynorm}, we rewrite $J_{\beta,\gamma}(t)$ in \eqref{Jbeta} as
 \begin{align*}
J_{\beta,\gamma}(t)
 &=  \int_0^t e^{-\beta_{-1}(t-\tau)} h_{\beta,\gamma}'(\tau) (A_{\C,\Lambda}+\beta_{-1}{\rm Id}_{P_{\Lambda,\C}H_\C})^{-1} e^{-(t-\tau)A_{\C,\Lambda}}P_{\Lambda,\C}   \xi_\beta\d\tau\\
  &=  \int_0^t e^{-\beta_{-1}(t-\tau)} h_{\beta,\gamma}'(\tau) (A_\C+\beta_{-1}{\mathbb I})^{-1} e^{-(t-\tau)A_\C} P_{\Lambda,\C} \xi_{\beta,\gamma}\d\tau.
\end{align*}
Recall that $\beta\in \mathcal E_\C(m,k,-\mu)$  and $m\ge 0$, hence $\Re(\beta_{-1})=0$.
This fact and  inequality \eqref{AZA} imply
 \begin{align*}
 |J_{\beta,\gamma}(t)|_{\alpha+1,\sigma}
& \le  \int_0^t |h_{\beta,\gamma}'(\tau)|\left|(A_\C+\beta_{-1}{\mathbb I})^{-1} e^{-(t-\tau)A_\C}P_{\Lambda,\C}   \xi_{\beta,\gamma}\right|_{\alpha+1,\sigma} \d\tau\\
&\le \int_0^t |h_{\beta,\gamma}'(\tau)|  \left|e^{-(t-\tau)A_\C}P_{\Lambda,\C}\xi_\beta\right|_{\alpha,\sigma}\d\tau
\le \int_0^t |h_{\beta,\gamma}'(\tau)| \cdot  e^{-(t-\tau)} |\xi_{\beta,\gamma}|_{\alpha,\sigma}\d\tau.
\end{align*}

\textit{Case $m=0$.} Applying estimate \eqref{dF1} to $h=h_{\beta,\gamma}$ and $s=\delta_*\in(0,1)$ and utilizing inequality \eqref{iine2} give
\beq\label{term30}
|J_{\beta,\gamma}(t)|_{\alpha+1,\sigma}\le C\int_0^t e^{-(t-\tau)} (T_*+\tau)^{-\mu-\delta_*}|\xi_{\beta,\gamma}|_{\alpha,\sigma}\d\tau \le C(T_*+t)^{-\mu-\delta_*}.
\eeq

\textit{Case $m\ge 1$.}  Applying   estimate \eqref{dF2}  to $h=h_{\beta,\gamma}$, $s=1/2$ and using inequality \eqref{iine2} give
\beqs
|J_{\beta,\gamma}(t)|_{\alpha+1,\sigma}
\le C\int_0^t e^{-(t-\tau)} (T_*+\tau)^{-1/2}|\xi_{\beta,\gamma}|_{\alpha,\sigma}\d\tau
 \le C(T_*+t)^{-1/2},
\eeqs
which yields
\beq\label{term31}
|J_{\beta,\gamma}(t)|_{\alpha+1,\sigma}
 \le C L_m(T_*+t)^{-\mu-\delta_*}.
\eeq

 Summing up in finitely many $\beta$ and $\gamma$  the inequalities \eqref{term30} and \eqref{term31}, one obtains, for both cases $m=0$ and $m\ge 1$,
\beq\label{term3all}
\sum_{(\beta,\gamma)\in S} |J_{\beta,\gamma}(t)|_{\alpha+1,\sigma}\le C L_m(T_*+t)^{-\mu-\delta_*}.
\eeq

Consider the last integral on the right-hand side of \eqref{ynorm}. We have the same estimate (5.24) in \cite{H6}, that is,
\beq\label{term4all}
\left|\int_0^t e^{-(t-\tau)A}P_{\Lambda} g(T_*+\tau)\d\tau\right|_{\alpha+1-\varep,\sigma}
\le C L_m(T_*+ t)^{-\mu-\delta_0}.
\eeq 

We fix  $\delta=\theta=1/2$ now. It follows the estimates \eqref{ynorm}, \eqref{term1},  \eqref{term3all}  and \eqref{term4all}, and the fact $\delta_*\le \delta_0$ that
\beq\label{preN}
\left |P_{\Lambda,\C}\big( w(T_*+t)-(\mathcal Z_{A_\C}p)(\Xi(T_*+t))\big)\right |_{\alpha+1-\varep,\sigma}\\
\le C L_m(T_*+t)^{-\mu-\delta_*} .
\eeq
Note that $w(T_*+t)\in G_{\alpha+1-\varep,\sigma}$  from part \ref{api}.
Also, the assumption  $p\in \widehat{\mathscr P}_m(k,\ell,-\mu,G_{\alpha,\sigma,\C})$ implies, thanks to Lemma \ref{invar4}\ref{Xa}\ref{ZAii}, 
 $\mathcal Z_{A_\C}p\in \widehat{\mathscr P}_m(k,\ell,-\mu,G_{\alpha+1,\sigma,\C})$, and, consequently, 
$(\mathcal Z_{A_\C}p)(\Xi(T_*+t))\in G_{\alpha+1,\sigma,\C}.$
Then we can pass $N\to\infty$ in \eqref{preN} to obtain, for $t\ge 1$,   
\beq\label{wrem2}
\left |w(T_*+t)-(\mathcal Z_{A_\C}p)(\Xi(T_*+t))\right |_{\alpha+1-\varep,\sigma}\le C L_m(T_*+t)^{-\mu-\delta_*}.
\eeq
Finally, by replacing $T_*+t$ with $t$ in \eqref{wrem2}, we obtain \eqref{wremain}.
The proof is complete.
\end{proof}

\section{Proofs of the main results}\label{proofsec}

\subsection{Proof of Theorem \ref{mainthm}}\label{proof1}

\begin{lemma}\label{Blem}
Let $\alpha\ge 1/2$, $\sigma\ge 0$, $k,m,\ell\in\Z_+$ with $k\ge m$, and $\mu,\mu'\in\R$. 
If $$p\in\widehat{\mathscr P}_{m}(k,\ell,\mu,G_{\alpha+1/2,\sigma,\C},G_{\alpha+1/2,\sigma})\text{ and }
q\in\widehat{\mathscr P}_{m}(k,\ell,\mu',G_{\alpha+1/2,\sigma,\C},G_{\alpha+1/2,\sigma}),$$
then 
 \beqs
 B_\C(p,q)\in \widehat{\mathscr P}_{m}(k,\ell,\mu+\mu',G_{\alpha,\sigma,\C},G_{\alpha,\sigma}).
 \eeqs
 \end{lemma}
\begin{proof}
 The proof is similar to that of Lemma \ref{prodreal} taking into account the second property in \eqref{BCbar} and inequality \eqref{BCas}. We omit the details.
\end{proof}

\begin{proof}[Proof of Proposition \ref{qregpower}] 
First we prove \eqref{qreg} by induction.
We present the proof for the case $m_*=0$, while the other case $m_*\ge 1$  can be proved  by simply taking $\chi_n=0$ in the proof below.

For $n=1$, we have  $p_1\in \widehat{\mathscr P}_{m_*}(M_1,\widetilde M_1,-\mu_1,G_{\alpha,\sigma,\C},G_{\alpha,\sigma})$, hence, 
$$q_1=\mathcal Z_{A_\C}p_1 \text{ belongs to }\widehat{\mathscr P}_{m_*}(M_1,\widetilde M_1,-\mu_1,G_{\alpha+1,\sigma,\C},G_{\alpha+1,\sigma})$$ thanks to Lemma \ref{invar4}\ref{ZAii}.
Thus, \eqref{qreg} is true for $n=1$.

Let $n\ge 2$. Suppose it holds for all $1\le k\le n-1$ that
\beq
\label{qreg3} 
q_k\in \widehat{\mathscr P}_{m_*}(M_k,\widetilde M_k,-\mu_k,G_{\alpha+1,\sigma,\C},G_{\alpha+1,\sigma}).
\eeq

For $\mu_m+\mu_j=\mu_n$ with $m,j\le n-1$, we have, by \eqref{qreg3} and the virtue of Lemma \ref{Blem}, $B_\C( q_m,q_j)$ belongs to
\beqs
\widehat{\mathscr P}_{m_*}(M_n,\widetilde M_n,-\mu_m-\mu_j,G_{\alpha+1/2,\sigma,\C},G_{\alpha+1/2,\sigma})
=\widehat{\mathscr P}_{m_*}(M_n,\widetilde M_n,-\mu_n,G_{\alpha+1/2,\sigma,\C},G_{\alpha+1/2,\sigma}).
\eeqs 
This implies
\beq \label{Breg3} 
B_\C( q_m,q_j) \in\widehat{\mathscr P}_{m_*}(M_n,\widetilde M_n,-\mu_n,G_{\alpha,\sigma,\C},G_{\alpha,\sigma}).
\eeq 

Connsidering  $\chi_n$ defined in Definition \ref{construct}, if it is zero then, certainly,
\beq\label{chireg3}
\chi_n\in\widehat{\mathscr P}_{m_*}(M_n,\widetilde M_n,-\mu_n,G_{\alpha,\sigma,\C},G_{\alpha,\sigma}).
\eeq 
Now consider the case $m_*=0$ and $\chi_n$ is given by the first formula in \eqref{chin}. 
Let  $\lambda\le n-1$ with $\mu_\lambda+1=\mu_n$. 
By \eqref{qreg3} for $k=\lambda$ and applying \eqref{pest4} to $p=q_\lambda$, $\mu=-\mu_\lambda$  and $\tilde X=(G_{\alpha,\sigma,\C},G_{\alpha,\sigma})$, 
we obtain the function 
$\chi_n=\mathcal R q_\lambda+\sum_{j=1}^{\widetilde M_\lambda} W_j\frac{\partial q_\lambda}{\partial \zeta_j}$ belongs to $\widehat{\mathscr P}_{m_*}(M_n,\widetilde M_n,-\mu_\lambda-1,G_{\alpha,\sigma,\C},G_{\alpha,\sigma})=\widehat{\mathscr P}_{m_*}(M_n,\widetilde M_n,-\mu_n,G_{\alpha,\sigma,\C},G_{\alpha,\sigma})$.
This fact yields \eqref{chireg3} again.
 
Now that all the terms 
$p_n$, $B_\C(q_m,q_j)$, $\chi_n$ in \eqref{qn} belong to $\widehat{\mathscr P}_{m_*}(M_n,\widetilde M_n,-\mu_n,G_{\alpha,\sigma,\C},G_{\alpha,\sigma})$, we have
$$ p_n-\sum_{\mu_m+\mu_j=\mu_n}B_\C(q_m,q_j)-\chi_n\in \widehat{\mathscr P}_{m_*}(M_n,\widetilde M_n,-\mu_n,G_{\alpha,\sigma,\C},G_{\alpha,\sigma}).$$
By applying $\mathcal Z_{A_\C}$ to this element and using Lemma \ref{invar4}\ref{ZAii}, we obtain \eqref{qreg}.
By the Induction Principle, we have \eqref{qreg} holds true for all $n\ge 1$.

Now that \eqref{qreg} is true for all $n\in\N$, the statement \eqref{Breg} follows \eqref{Breg3}, and  \eqref{chireg} follows \eqref{chireg3}.
\end{proof}

We are now ready to prove Theorem \ref{mainthm}.

\begin{proof}[Proof of Theorem \ref{mainthm}]
Define functions $\Xi_n$, for $n\in\N$, by 
\beqs
\Xi_n(t)=(\widehat{\LL}_{M_n}(t),\widehat{\mathcal Y}_{\widetilde M_n}(t))\text{ for sufficiently large $t>0$.}
\eeqs
For $n\in\N$, let 
\begin{align*}
u_n(t)&=q_n(\Xi_n(t)),\
 U_n(t)=\sum_{j=1}^n u_j(t) \text{ and } v_n=u(t)-U_n(t),\\
f_n(t)&=p_n(\Xi_n(t)),\
F_n(t)=\sum_{j=1}^n f_j(t) \text{ and } g_n(t)=f(t)-F_n(t).
\end{align*}
Denote $\psi(t)=L_{m_*}(t)$. To establish the asymptotic expansion \eqref{uexpand}, it suffices to prove that for any $N\in \N$, there exists $\delta_N>0$ such that 
\beq\label{PTN}
\Big|u(t)-\sum_{n=1}^N u_n(t) \Big|_{\alpha+1-\rho,\sigma}=\bigo(\psi(t)^{-\mu_N-\delta_N}) \text{ for all }\rho\in(0,1).
\eeq

First, we derive some basic information for the solution $u(t)$ and the force $f(t)$. 
It follows expansion \eqref{fseq} that  
\beq \label{Fcond}
|g_n(t)|_{\alpha,\sigma}=\mathcal O(\psi(t)^{-\mu_n-\varep_n}) ,
\eeq
for any $n\in\N$, with some $\varep_n>0$.
For any $n\in\N$ and $\delta>0$, we have, by the virtues of Proposition \ref{qregpower} and \eqref{LLo}, 
\begin{align}
\label{fnrate}
|f_n(t)|_{\alpha+1,\sigma}&=\mathcal O(\psi(t)^{-\mu_n+\delta}),\\
\label{unrate}
|u_n(t)|_{\alpha+1,\sigma}&=\mathcal O(\psi(t)^{-\mu_n+\delta}),\\
\label{ubarate}
|U_n(t)|_{\alpha+1,\sigma}&=\mathcal O(\psi(t)^{-\mu_1+\delta}).
\end{align}

Fix a real number $T_*$ such that  $T_*\ge E_{m_*+1}(0)$ and $u(t)$ is a Leray--Hopf weak solution on $[T_*,\infty)$. Note that $\psi(t+T_*)\ge 1$ for $t\ge 0$.
For $\delta>0$ and sufficiently large $t$, we have, by  \eqref{Fcond} and \eqref{fnrate}, 
\beq\label{fest}
|f(t)|_{\alpha,\sigma}\le |f_1(t)|_{\alpha,\sigma}+|g_1(t)|_{\alpha,\sigma}
=\bigo(\psi(t)^{-\mu_1+\delta} )+\mathcal O(\psi(t)^{-\mu_1-\varep_1}) =\mathcal O(\psi(T_*+t)^{-\mu_1+\delta}),
\eeq
with the last estimate coming from \eqref{Lshift}. 

Given any number $\delta\in(0,\mu_1)$.  Define the function $F(t)=\psi(T_*+t)^{-\mu_1+\delta}$ for $t\ge 0$. 
By \eqref{fest}, $f$ satisfies the condition \eqref{falphaonly}.
As a result of Theorem \ref{Fthm2} applied to  $\varep=1/2$,
 there exists $\hat{T}>0$ and a constant $C>0$ 
such that $u(t)$ is a regular solution |on $[\hat{T},\infty)$, and, by estimate \eqref{us0}, 
 \beq\label{ups0}
 |u(\hat{T}+t)|_{\alpha+1/2,\sigma} \le C\psi(T_*+t)^{-\mu_1+\delta}  \text{ for all }t\ge 0.
 \eeq
Combining \eqref{AalphaB} with \eqref{ups0} gives 
 \beqs
|B(u(\hat{T}+t),u(\hat{T}+t))|_{\alpha,\sigma}\le  c_*^\alpha|u(\hat{T}+t)|_{\alpha+1/2,\sigma}^2 
\le C^2c_*^\alpha \psi(T_*+t)^{-2\mu_1+2\delta} \text{ for all }t\ge 0.
\eeqs
Together with \eqref{Lshift}, this estimate yields
\beq\label{buuest}
|B(u(t),u(t)))|_{\alpha,\sigma}=\bigo(\psi(t)^{-2\mu_1+2\delta}).
\eeq

We now prove \eqref{PTN} by induction in $N$.
We consider cases $m_*=0$ and $m_*\ge 1$ separately.

\medskip
\noindent\textbf{Case I:  $m_*=0$.}
In calculations below, all differential equations hold in $V'$-valued distribution sense on $(T,\infty)$, which is similar to \eqref{varform}, for any sufficiently $T>0$.
One can easily verify them by using \eqref{Bweak}, and the facts $u\in L^2_{\rm loc}([0,\infty),V)$ and $u'\in L^1_{\rm loc}([0,\infty),V')$ in Definition \ref{lhdef}.

\medskip
\textit{The first case $N=1$.} 
Rewrite the NSE \eqref{fctnse} as
\beq\label{uH1eq}
u'+Au=-B(u,u)+f_1+(f-f_1)=f_1+H_1,
\eeq
where
$H_1=-B(u,u)+g_1.$
By \eqref{Fcond} and taking $\delta=\mu_1/4$ in \eqref{buuest}, we obtain
\beq\label{H1est}
|H_1(t)|_{\alpha,\sigma}=\bigo(\psi(t)^{-\mu_1-\delta^*_1}),\text{ where } \delta^*_1=\min\{\varep_1,\mu_1/2\}>0.
\eeq

Note that $q_1=\mathcal Z_{A_\C}p_1$ and $u_1(t)= q_1(\widehat{\LL}_{M_1}(t))$. By equation \eqref{uH1eq} and estimate \eqref{H1est}, we can apply Theorem \ref{Fode}  to $w=u$, $p=p_1$, $k=M_1$, $\mu=\mu_1$ and $g=H_1$. Then there exists $\delta_1>0$ such that
\beqs
|u(t)-u_1(t)|_{\alpha+1-\rho,\sigma}=\bigo(\psi(t)^{-\mu_1-\delta_1}) \text{ for all }\rho\in(0,1).
\eeqs

Thus, \eqref{PTN} is true for $N=1$.

\medskip
\textit{The induction step.} Given any number $N \in \N$, assume there exists $\delta_N>0$ such that
\beq\label{vNrate}
|v_N(t)|_{\alpha+1-\rho,\sigma}=\mathcal O(\psi(t)^{-\mu_N-\delta_N})\text{ for all }\rho \in (0,1). 
\eeq
We aim to apply the key approximation result - Theorem \ref{Fode} - to  $v_N(t)$. For that, we find an appropriate differential equation for $v(t)$ starting with
\beqs
v_N'
=u'-\sum_{k=1}^N u_k'
=-Au -B(u,u)+f - \sum_{k=1}^N u_k'.
\eeqs 
Rewriting
\beqs
Au=\sum_{k=1}^N Au_k+Av_N\text{ and }
f= \sum_{k=1}^N  f_k+f_{N+1}+g_{N+1},
\eeqs
 we have
\beq\label{vN4}
v_N'
= -Av_N-B(u,u)+f_{N+1}- \sum_{k=1}^N \Big(Au_k+u_k' -f_k\Big)+g_{N+1}.
\eeq 

We recalculate the right-hand side of \eqref{vN4}.
First, we split nonlinear term $B(u,u)$  into 
\beq\label{Buu}
B(u,u)=B(U_N+v_N,U_N+v_N)
=B(U_N,U_N)+h_{N+1,1},
\eeq
 where
\beq\label{hN1}
h_{N+1,1}=B(U_N,v_N)+B(v_N,U_N)+B(v_N,v_N).
\eeq
More explicitly, 
\beq\label{BUN}
B(U_N,U_N) =\sum_{ m,j=1}^N B(u_m,u_j).
\eeq

Set $\mathcal S=\{\mu_n:n\in\N\}$.
According to Assumption \ref{B1}, $\mathcal S$ preserves the addition. Hence, the sum 
$\mu_m+\mu_j\in \mathcal S$, which yields $\mu_m+\mu_j=\mu_k$ for some $k\ge 1$.
With this, we split the sum in \eqref{BUN} into two parts:
$ \mu_m+\mu_j=\mu_k$  for 
$k\le N+1$ and for $k\ge N+2$.
It results in
\beq\label{BUN2}
B(U_N,U_N)
=\sum_{k=1}^{N+1}\Big( \sum_{\substack{1\le m,j\le N,\\ \mu_m+\mu_j=\mu_k}}B(u_m,u_j)\Big) +h_{N+1,2},
\eeq
where
\beq\label{hN2}
h_{N+1,2}=\sum_{\substack{1\le m,j\le N,\\\mu_m+\mu_j\ge \mu_{N+2}}} B(u_m,u_j).
\eeq

For $1\le k\le N+1$, define
\beq\label{Jk}
\mathcal J_k(t)=\sum_{\mu_m+ \mu_j=\mu_k} B(u_m(t),u_j(t)).
\eeq
It is clear in \eqref{Jk} that $m,j<k\le N+1$, hence, $m,j\le N$. 
Combining \eqref{Buu} and \eqref{BUN2} yields
\beq\label{Gy2}
B(u(t),u(t))
=\sum_{k=1}^{N+1}\mathcal J_k(t) +h_{N+1,1}(t)+h_{N+1,2}(t).
\eeq  

We now turn to the summation in $k$ in \eqref{vN4}.
For the derivative $u_k'$, we apply formula \eqref{dpYL2} to $p=q_k$ have
\beq \label{ykeq2}    
u_k'=\left(\mathcal M_{-1}q_k+\mathcal R q_k  + \sum_{j=1}^{\widetilde M_k} W_j \frac{\partial q_k}{\partial \zeta_j} \right)\circ \Xi_k.
\eeq  
Summing up the formula \eqref{ykeq2} in $k$ gives
\beq\label{shorty}
\sum_{k=1}^N u_k'=\sum_{k=1}^N \mathcal M_{-1}q_k\circ \Xi_k + \sum_{\lambda=1}^N \left(\mathcal R q_\lambda+ \sum_{j=1}^{\widetilde M_\lambda} W_j\frac{\partial q_\lambda}{\partial \zeta_j} \right)\circ \Xi_\lambda.
\eeq
In  \eqref{shorty} above, we changed the index $k$ from \eqref{ykeq2} to $\lambda$ for its last summation.
It follows \eqref{qreg} and \eqref{pest4}  that 
\beq\label{Rqprop} 
\mathcal Rq_\lambda+ \sum_{j=1}^{\widetilde M_\lambda} \frac{\partial q_\lambda}{\partial \zeta_j} W_j\in \widehat{\mathscr P}_{0}(M_\lambda,\widetilde M_\lambda,-\mu_\lambda-1,G_{\alpha,\sigma,\C},G_{\alpha,\sigma}).
\eeq 
According to Assumption \ref{B1} again with $m_*=0$,  we have $\mu_\lambda+1\in\mathcal S$.  Then there exists a unique number $k\in\N$ such that $\mu_k=\mu_\lambda+1$. 
By definition \eqref{chin}, 
$$\mathcal Rq_\lambda+ \sum_{j=1}^{\widetilde M_\lambda}  W_j\frac{\partial q_\lambda}{\partial \zeta_j}=\chi_k.$$
Considering two possibilities $k\le N+1$ and $k\ge N+2$, we  rewrite, similar to \eqref{Gy2},
\beq\label{sumR}
\sum_{\lambda=1}^N \left(\mathcal R q_\lambda+ \sum_{j=1}^{\widetilde M_\lambda} W_j\frac{\partial q_\lambda}{\partial \zeta_j} \right)\circ \Xi_\lambda=\sum_{k=1}^{N+1} \chi_k\circ \Xi_k +h_{N+1,3}(t),
\eeq
where
\beq\label{hN3}
h_{N+1,3}=\sum_{\substack{1\le \lambda\le N,\\ \mu_\lambda+1\ge \mu_{N+2}}}  \left(\mathcal R q_\lambda+ \sum_{j=1}^{\widetilde M_\lambda} W_j\frac{\partial q_\lambda}{\partial \zeta_j} \right)\circ \Xi_\lambda.
\eeq

Combining \eqref{vN4}, \eqref{Gy2}, \eqref{shorty} and \eqref{sumR} yields 
\beq\label{vN3}
v_N'+Av_N
=- \mathcal J_{N+1}(t) + f_{N+1}(t) - \sum_{k=1}^N X_k(t) - \chi_{N+1}\circ \Xi_{N+1}(t)+H_{N+1}(t),
\eeq 
where
\begin{align} 
\label{Xk}
X_k(t)&=\mathcal J_k(t)+ (Aq_k+\mathcal M_{-1}q_k + \chi_k-p_k)\circ \Xi_k(t),\\
\label{HN1}
H_{N+1}(t)&=-h_{N+1,1}(t)-h_{N+1,2}(t)-h_{N+1,3}(t)+g_{N+1}(t).
\end{align}

For $k\in\N$, define 
\beq\label{Qk}
\mathcal Q_k=\sum_{\mu_m+\mu_j=\mu_k}   B(q_m,q_j)\text{ which equals }\sum_{\mu_m+\mu_j=\mu_k}   B_\C(q_m,q_j). 
\eeq
Indeed, the last equality comes from the fact that 
\beq\label{qreal}
q_m(z,\zeta),q_j(z,\zeta)\in G_{\alpha,\sigma}.
\eeq 
Thanks to \eqref{Breg}, one has
\beq\label{Qreg}
\mathcal Q_k\in \widehat{\mathscr P}_{m_*}(M_{N+1},\widetilde M_{N+1},-\mu_k,G_{\alpha,\sigma,\C},G_{\alpha,\sigma}).
\eeq
It is clear from definition \eqref{Jk} and \eqref{Qk} that  
$$\mathcal Q_k(\Xi_k(t))=\mathcal J_k(t),\text{ that is, } \mathcal J_k=\mathcal Q_k\circ \Xi_k \text{ for all }k\in\N.$$
Considering $X_k$ given by \eqref{Xk}, we use \eqref{qreal} again to rewrite
\beqs
Aq_k+\mathcal M_{-1}q_k=(A_\C+\mathcal M_{-1})q_k.
\eeqs
The other terms can be rewritten, thanks to  \eqref{ZAM}, by
\beqs
\mathcal J_k=\mathcal Q_k\circ \Xi_k=((A_\C+\mathcal M_{-1})\mathcal Z_{A_\C}\mathcal Q_k)\circ \Xi_k,\quad 
\chi_k-p_k=(A_\C+\mathcal M_{-1})\mathcal Z_{A_\C}(\chi_k-p_k).
\eeqs
Therefore, we obtain
\beqs
X_k(t)=\Big [(A_\C+\mathcal M_{-1})\Big( q_k + \mathcal Z_{A_\C}(\mathcal Q_k+\chi_k-p_k)\Big)\Big] \circ \Xi_k(t).
\eeqs 

For $1\le k\le N$, one has from \eqref{qn} that
$q_k = \mathcal Z_{A_\C}(p_k-\mathcal Q_k-\chi_k )$,
which implies $X_k=0$.
Now that the sum $\sum_{k=1}^N X_k(t)$ is zero in  equation \eqref{vN3}, we use the facts $f_{N+1}=p_{N+1}\circ \Xi_{N+1}$ and $\mathcal J_{N+1}=\mathcal Q_{N+1}\circ \Xi_{N+1}$ to obtain 
\beq\label{vN5}
v_N'(t)+Av_N(t) 
=(- \mathcal Q_{N+1}+p_{N+1}- \chi_{N+1})\circ \Xi_{N+1} (t) 
 +H_{N+1}(t).
\eeq

For the first term on the right-hand side of \eqref{vN5}, we have, 
thanks to the property of $\mathcal Q_{N+1}$ from \eqref{Qreg} for $k=N+1$, 
the property of  $p_{N+1}$ from  \eqref{pncond} for $n=N+1$, 
and the property of $\chi_{N+1}$ from \eqref{chireg}, that
\beq\label{preg}
- \mathcal Q_{N+1}+p_{N+1}- \chi_{N+1}\in \widehat{\mathscr P}_{m_*}(M_{N+1},,\widetilde M_{N+1},-\mu_{N+1},G_{\alpha,\sigma,\C},G_{\alpha,\sigma}).
\eeq

We estimate $H_{N+1}(t)$ next
by finding bounds for the $G_{\alpha,\sigma}$-norm of each term on the right-hand side of \eqref{HN1}. 
For the term $h_{N+1,1}(t)$, which given by \eqref{hN1}, we use property \eqref{ubarate}, estimate \eqref{vNrate} with $\rho=1/2$,  and inequality \eqref{AalphaB} to obtain, for any $\delta>0$,
\begin{align*}
|B(U_N(t),v_N(t))|_{\alpha,\sigma},|B(v_N(t),U_N(t))|_{\alpha,\sigma}
&=\bigo(|U_N(t)|_{\alpha+1/2,\sigma}|v_N(t)|_{\alpha+1/2,\sigma})\\
&=\bigo(\psi(t)^{-\mu_1+\delta-\mu_N-\delta_N}),\\
|B(v_N(t),v_N(t))|_{\alpha,\sigma}
&=\bigo(|v_N(t)|_{\alpha+1/2,\sigma}^2)
=\bigo(\psi(t)^{-2\mu_N-2\delta_N}).
\end{align*}
These estimates yield
\beq\label{hh1}
|h_{N+1,1}(t)|_{\alpha,\sigma}=\bigo(\psi(t)^{-\mu_1-\mu_N+\delta-\delta_N})+\bigo(\psi(t)^{-2\mu_N-2\delta_N}).
\eeq
Because $2\mu_N,\mu_1+\mu_N \in\mathcal S$ and $2\mu_N,\mu_1+\mu_N>\mu_N$, we have $2\mu_N,\mu_1+\mu_N\ge \mu_{N+1}$. Then taking $\delta=\delta_N/2$ in \eqref{hh1} gives
\beq\label{h1e}
|h_{N+1,1}(t)|_{\alpha,\sigma}
= \mathcal O(\psi(t)^{-\mu_{N+1}-\delta_N/2}).
\eeq

For the function  $h_{N+1,2}(t)$, which is given in \eqref{hN2},  we utilize inequalities \eqref{AalphaB} and   \eqref{unrate}, noticing that   $\mu_m+\mu_j=\mu_k\ge \mu_{N+2}$, to derive
\beqs 
|B( u_m(t), u_j(t))|_{\alpha,\sigma}
=\bigo(|u_m(t)|_{\alpha+1/2,\sigma}|u_j(t)|_{\alpha+1/2,\sigma})
=\bigo(\psi(t)^{-\mu_m-\mu_j+\delta })
=\bigo(\psi(t)^{-\mu_{N+2}+\delta })
\eeqs
for any $\delta>0$.
Taking $\delta=\delta_{N+1}'\eqdef (\mu_{N+2}-\mu_{N+1})/2$, and summing up in $m$ and $j$ finitely many times,  we obtain
\beq\label{h2e}
|h_{N+1,2}(t)|_{\alpha,\sigma}=\bigo(\psi(t)^{-\mu_{N+1}-\delta_{N+1}'}).
\eeq

Regarding $h_{N+1,3}(t)$, we  have from property \eqref{Rqprop}  for each summand of $h_{N+1,3}$ in \eqref{hN3}  and the fact $\mu_\lambda+1=\mu_k\ge \mu_{N+2}$ that
\beq\label{h3e}
|h_{N+1,3}(t)|_{\alpha,\sigma}=\bigo(\psi(t)^{-\mu_{N+2}+\delta_{N+1}'})=\bigo(\psi(t)^{-\mu_{N+1}-\delta_{N+1}'}).
\eeq

For $g_{N+1,3}(t)$, estimate   \eqref{Fcond} for $n=N+1$ gives
\beq\label{ge}
|g_{N+1}(t)|_{\alpha,\sigma}=\bigo(\psi(t)^{-\mu_{N+1}-\varep_{N+1}}).
\eeq

Combining \eqref{HN1} with the estimates \eqref{h1e}, \eqref{h2e}, \eqref{h3e} and \eqref{ge}  above gives
\beq\label{HN1est}
|H_{N+1}(t)|_{\alpha,\sigma}=\bigo(\psi(t)^{-\mu_{N+1}-\delta^*_{N+1}}),
\eeq 
where $\delta^*_{N+1}=\min\{\varep_{N+1},\delta_N/2,\delta_{N+1}'\}>0$.

Based on equation \eqref{vN5}, property \eqref{preg} and estimate \eqref{HN1est}, we  apply  Theorem \ref{Fode} to 
$w=v_N$, $p= p_{N+1}-\mathcal Q_{N+1} -\chi_{N+1}$, $k=M_{N+1}$, $\mu=\mu_{N+1}$ and $g=H_{N+1}$.
Then there exists $\delta_{N+1}>0$ such that 
\beqs
\left|v_N(t)-(\mathcal Z_{A_\C} (  -\mathcal Q_{N+1} +p_{N+1}-\chi_{N+1}))\circ \Xi_{N+1}(t)\right|_{\alpha+1-\rho,\sigma}=\bigo(\psi(t)^{-\mu_{N+1}-\delta_{N+1}})
\eeqs
 for all $\rho\in(0,1)$. According to the definitions \eqref{qn} and \eqref{Qk}, we have
\beqs
(\mathcal Z_{A_\C} (-\mathcal Q_{N+1}  p_{N+1} -\chi_{N+1}))\circ \Xi_{N+1}=q_{N+1}\circ \Xi_{N+1}=u_{N+1},
\eeqs
hence
\beqs
|v_N(t)-u_{N+1}(t)|_{\alpha+1-\rho,\sigma}=\bigo(\psi(t)^{-\mu_{N+1}-\delta_{N+1}})\text{ for all }\rho\in(0,1).
\eeqs
This, together with the fact $v_N-u_{N+1}=u-\sum_{n=1}^{N+1}u_n$,  implies \eqref{PTN} for $N:=N+1$.

\medskip
\textit{Conclusion for Case I.} By the induction principle, we have \eqref{PTN} holds true for all $N\in\N$.
The proof for the case $m_*=0$ is complete.

\medskip
\noindent\textbf{Case II: $m_*\ge 1$.}
We repeat the  proof in Case I with the following appropriate modifications.
In \eqref{shorty}, we have, thanks to  \eqref{pest8} of Lemma \ref{Wlem}, 
$$\mathcal R q_\lambda+\sum_{j=1}^{\widetilde M_\lambda}  W_j\frac{\partial q_\lambda}{\partial \zeta_j} \in \widehat{\mathscr P}_{0}(M_\lambda,\widetilde M_\lambda,-1, G_{\alpha,\sigma,\C},G_{\alpha,\sigma})$$
Hence, we can estimate  the last summation in $\lambda$ in  \eqref{shorty} by
\beqs
\left|\sum_{\lambda=1}^N \left(\mathcal R q_\lambda+ \sum_{j=1}^{\widetilde M_\lambda} W_j \frac{\partial q_\lambda}{\partial \zeta_j} \right)\circ \Xi_\lambda)(t)\right|_{\alpha,\sigma}
=\bigo(t^{-s}), \text{ for any $s\in(0,1)$,}
\eeqs
which is of $\bigo(\psi(t)^{-\mu_{N+1}-\delta})$ for any $\delta>0$. Then we ignore \eqref{sumR} and can take $\chi_k=0$ in all computations after that.
The end result is \eqref{PTN} holds true for all $N\in\N$ with $q_n$ being defined by  \eqref{qn} with  $\chi_n=0$. 
\end{proof}

\subsection{Proof of Theorem \ref{thm3}}\label{proof2}

The proof of  Theorem \ref{thm3} will use the following conversions between the functions in 
Sections \ref{classtype} and in subsection \ref{Rsec}.

\begin{lemma}\label{convert}
 Let $X$ be a linear space over $\R$, and $X_\C$ be its complexification.
Given integers $k\ge m\ge 0$ and $\ell\ge 0$. 
\begin{enumerate}[label=\tnum]

\item\label{cv1} If $p\in \widehat{\mathcal P}_{m}(k,\ell,X) $, then
\beq\label{spec1}
 p\left (\widehat{\LL}_k(t),\zeta\right)
= q\left (\widehat{\LL}_k(t),\zeta\right)  \text{ for some } q\in \widehat{\mathscr P}_{m}(k,\ell,0,X_\C,X).
\eeq
More specifically, the function $q(z,\zeta)$ in \eqref{spec1}  is of the form
\beq\label{spec2}
q(z,\zeta)=\sum_{(\beta,\gamma)\in S} z^\beta\zeta^\gamma \xi_{\beta,\gamma} \text{ as in Definition \ref{realH} with }\beta=(\beta_{-1},\beta_0,\ldots,\beta_{k}),\Im(\beta_{k})=0. 
\eeq

\item\label{cv2} Suppose $q\in \widehat{\mathscr P}_{m}(k,\ell,0,X_\C,X) $. Then
\beq\label{spec3}
q\left (\widehat{\LL}_k(t),\zeta\right) 
=p\left (\widehat{\LL}_{k+1}(t),\zeta\right) 
\text{ for some } p\in \widehat{\mathcal P}_{m}(k+1,\ell,X).
\eeq
Moreover, if $q(z,\zeta)$ is of the form \eqref{spec2},
then
\beq \label{spec4}
p\in \widehat{\mathcal P}_{m}(k,\ell,X).
\eeq

\item \label{cv3}
If $p\in\widehat{\mathcal P}_m^{+}(k,\ell)$, then 
\beq\label{spec0}
p\left (\widehat{\LL}_k(t),\zeta\right)=q\left(\widehat{\LL}_k(t),\zeta\right) \text{ for some function } 
q\in \widehat{\mathscr P}_m^{+}(k,\ell). 
\eeq
Moreover, $q(z,\zeta)$ has the form \eqref{spec2}.
\end{enumerate}
\end{lemma}
\begin{proof}
 Part  \ref{cv1}. 
 It suffices to consider
 \beq\label{simp}
p(z,\zeta)=z^\beta \zeta^\gamma  \left(\prod_{j=1}^{j_*} \sigma_j(\omega_j z_{\kappa_j})\right) 
\left(\prod_{j=1}^{\tilde j_*} \tilde\sigma_j(\tilde\omega_j \ln\zeta_{\tilde \kappa_j}) \right)  \xi \in \widehat{\mathcal P}_{m}(k,\ell,X)
\eeq
as in Definition \ref{realPL}. Recall that $\beta\in \mathcal E_\R(m,k,0)\subset \R^{k+2}$ and $\gamma\in \R^\ell$.
 We have the function
    \begin{equation}\label{realgood}
       Q_0(z,\zeta)\eqdef z^\beta \zeta^\gamma\xi\text{ belongs to }\widehat{\mathscr P}_{m}(k,\ell,0,\C,\R).
    \end{equation}
 Observe, for  $\omega\in \R$ and $0\le j\le k$,  that
 \beq\label{sincos1}
  \cos(\omega\iln_j(t))=g(\widehat{\LL}_{k}(t),\zeta ) 
  \text{ and }
  \sin(\omega\iln_j(t))=h(\widehat{\LL}_{k}(t),\zeta ) ,
   \eeq
 where
 \beq\label{sincos2} 
 g(z,\zeta)=\frac12(z_{j-1}^{i\omega}+z_{j-1}^{-i\omega})\text{ and }
h(z,\zeta)=\frac1{2i}(z_{j-1}^{i\omega}-z_{j-1}^{-i\omega}).
\eeq
Similarly, for  $\tilde \omega\in \R$ and $0\le j\le \ell$, one has
\beq\label{scze1}
  \cos(\tilde\omega\ln\zeta_j)=\tilde g(z,\zeta)\text{ and }\sin(\tilde\omega\ln\zeta_j)=\tilde h(z,\zeta),
   \eeq
 where
 \beqs 
 \tilde g(z,\zeta)=\frac12(\zeta_j^{i\tilde\omega}+\zeta_j^{-i\tilde\omega})\text{ and }
\tilde h(z,\zeta)=\frac1{2i}(\zeta_j^{i\tilde\omega}-\zeta_j^{-i\tilde\omega}).
\eeqs
Obviously,
\beq\label{sincos3} 
g,h,\tilde g,\tilde h\in \widehat{\mathscr P}_m(k,\ell,0,\C,\R).
\eeq
To calculate $p(\widehat{\LL}_k(t),\zeta)$ we multiply $Q_0(\widehat{\LL}_k(t),\zeta)$ with the sinusoidal functions in  \eqref{sincos1} and \eqref{scze1}. It results in  
$p(\widehat{\LL}_k(t),\zeta)= q(\widehat{\LL}_k(t),\zeta)$,
where $q(z,\zeta)$ is a product of the function $Q_0(z,\zeta)$ with finitely many functions of types $g$, $h$, $\tilde g$, $\tilde h$ above. By \eqref{realgood},  \eqref{sincos3} and applying Lemma  \ref{prodreal} to the product, 
we have $q\in \widehat{\mathscr P}_{m}(k,\ell,0,X_\C,X)$. Thus, we obtain \eqref{spec1}.

For the second statement, we observe that, in the above process of obtaining $q(z,\zeta)$, the possible appearance  of $z_k^{i\Im (\beta_k)}$ can only come from \eqref{sincos2}. However, each  variable $z_{j-1}$ there has the index $j-1<k$.
 Therefore, the resulted formula of $q(z,\zeta)$ does not contain $z_k^{i\Im (\beta_k)}$ with $\Im(\beta_k)\ne 0$, which can be said equivalently that $\Im(\beta_k)=0$ in \eqref{spec2}.

\medskip
Part \ref{cv2}.  
It suffices to consider, for $z\in (0,\infty)^{k+2}$ and $\zeta\in(0,\infty)^\ell$, 
\beq \label{qqstar}
q(z,\zeta)=q_*(z,\zeta)+\overline {q_*(z,\zeta)},\text{ where } q_*(z,\zeta)=z^\beta\zeta^\gamma\xi
\eeq 
with $\xi\in X_\C$, $\gamma\in \C^\ell$, and
\beq\label{recallbx} 
\beta\in\C^{k+2}\text{ having } \Re\beta\in \mathcal E_\R(m,k,0).
\eeq
Denote $\omega=\Im\beta=(\omega_{-1},\omega_0,\ldots,\omega_k)\in \R^{k+2}$ and $\tilde \omega=\Im\gamma=(\tilde \omega_1,\ldots,\tilde \omega_\ell)\in\R^\ell$.
Then
\begin{equation*}
\notag
\widehat{\LL}_k(t)^\beta \zeta^\gamma
=\widehat{\LL}_k(t)^{\Re\beta} \zeta^{\Re\gamma}
\left( \prod_{j=-1}^k L_j(t)^{i\omega_j}\right)\left(\prod_{j=1}^\ell  \zeta_j^{i\tilde\omega_j}\right),
\end{equation*}
which we rewrite as
\beq \label{mess}
\begin{aligned}
\widehat{\LL}_k(t)^\beta \zeta^\gamma
&=\widehat{\LL}_k(t)^{\Re\beta} \zeta^{\Re\gamma}
\left\{  \prod_{j=-1}^k\left [\cos(\omega_j L_{j+1}(t))+i\sin(\omega_jL_{j+1}(t))\right] \right\}\\
&\quad \times\left\{ \prod_{j=1}^\ell\left [\cos(\tilde\omega_j\ln \zeta_j)+i\sin(\tilde\omega_j\ln \zeta_j)\right] \right\}.
\end{aligned}
\eeq 
Expanding all the products in \eqref{mess}, we obtain
\beqs
\widehat{\LL}_k(t)^\beta \zeta^\gamma
=\widehat{\LL}_k(t)^{\Re\beta} \zeta^{\Re\gamma}\left\{p_1(z,\zeta)\Big|_{z=\widehat{\LL}_{k+1}(t)}+ip_2(z,\zeta)\Big|_{z=\widehat{\LL}_{k+1}(t)}\right\},
\eeqs
for some functions $p_1,p_2 : (0,\infty)^{k+3}\times (0,\infty)^\ell\to \R$, each is a finite sum of the functions of  the form
\beq\label{ppbase}
(z,\zeta)\mapsto  \left(\prod_{j=0}^{k+1} \sigma_j(\omega_{j-1} z_j)\right) 
\left(\prod_{j=1}^\ell \tilde\sigma_j(\tilde\omega_j \ln\zeta_j) \right)\text{ with functions }\sigma_j,\tilde\sigma_j\in\{\cos,\sin\}.
\eeq
 We then have 
\begin{align*}
q_*(\widehat{\LL}_k(t),\zeta)
&=\widehat{\LL}_k(t)^\beta \zeta^\gamma (\Re\xi+i\Im\xi)\\
&=\widehat{\LL}_k(t)^{\Re\beta} \zeta^{\Re\gamma} 
\left\{\left [p_1(\widehat{\LL}_{k+1}(t),\zeta)\Re\xi - p_2(\widehat{\LL}_{k+1}(t),\zeta)\Im\xi\right]\right.\\
&\quad +i\left.\left [p_1(\widehat{\LL}_{k+1}(t),\zeta)\Im\xi + p_2(\widehat{\LL}_{k+1}(t),\zeta)\Re\xi\right]\right\}.
\end{align*} 
Therefore,
\beq\label{realpart}
q_*(\widehat{\LL}_k(t),\zeta)
=2 \Re\left[q_*(\widehat{\LL}_k(t),\zeta)\right]
=p(\widehat{\LL}_{k+1}(t),\zeta),
\eeq 
where, for $z=(z_{-1},z_0,\ldots,z_k,z_{k+1})\in (0,\infty)^{k+3}$,
\beq \label{qzz}
p(z,\zeta)= 2(z_{-1},z_0,\ldots,z_k)^{\Re\beta} \zeta^{\Re\gamma}\left[ p_1(z,\zeta)\Re\xi - p_2(z,\zeta)\Im\xi\right].
\eeq 
Combing formula \eqref{qzz} with \eqref{recallbx} and the fact $p_1,p_2$ being finite sums of functions from \eqref{ppbase}, we obtain
 $p\in \widehat{\mathcal P}_{m}(k+1,\ell,X)$.
Thus, we have \eqref{spec3} from \eqref{realpart}.

For the second statement, it suffices to consider $q(z,\zeta)$ to be as in \eqref{qqstar} again. We follow the above calculations from \eqref{qqstar} to \eqref{qzz}. By the last assumption in \eqref{spec2}, one has $\omega_k=\Im(\beta_k)=0$. 
This deduces that  all the terms $\cos(\omega_kL_{k+1}(t))=1$ and $\sin(\omega_kL_{k+1}(t))=0$ in \eqref{mess}.
Hence, in constructing the functions $p_1,p_2$, the product $\prod_{j=0}^{k+1}$ in \eqref{ppbase} can be replaced with $\prod_{j=0}^{k}$, that is,  \eqref{ppbase}  does not have the variable $z_{k+1}$ in this situation. 
Also, the part $(z_{-1},z_0,\ldots,z_k)^{\Re\beta} \zeta^{\Re\gamma}$ in formula \eqref{qzz} of $p(z,\zeta)$ does not contain the variable  $z_{k+1}$.
Therefore, $p\in \widehat{\mathcal P}_{m}(k,\ell,X) $.

\medskip  
Part \ref{cv3}. Assume $p\in\widehat{\mathcal P}_m^{+}(k,\ell)$ is $p(z,\zeta)=p_*(\zeta)+p_0(z,\zeta)$
as in \eqref{prealde} with $p_0\in \widehat{\mathcal P}_m(k,\ell,\R)$ playing the role of $q$ in \eqref{prealde}. By definition, $p_*(\zeta)$  is the same as in \eqref{pdecomp}.
By part \ref{cv1}, there is $q_0\in \widehat{\mathscr P}_{m}(k,\ell,0,\C,\R)$ such that
\beq\label{rspec}
 p_0 (\widehat{\LL}_k(t),\zeta)
= q_0 (\widehat{\LL}_k(t),\zeta)
\eeq
and $q_0$ satisfies \eqref{spec2}.
Define $q(z,\zeta)=p_*(\zeta)+q_0(z,\zeta)$.
We then have $q\in  \widehat{\mathscr P}_{m}(k,\ell,0,\C,\R)$.
Thanks to \eqref{rspec}, we have $p(\widehat{\LL}_k(t),\zeta)=q(\widehat{\LL}_k(t),\zeta)$, which proves the first part of \eqref{spec0}. 
Because of the particular form of $p_*(\zeta)$, which is independent of $z$, and the fact $q_0(z,\zeta)$ satisfies \eqref{spec2},  the functions $q(z,\zeta)$ also has property \eqref{spec2}.

It remains to prove that $q\in \widehat{\mathscr P}_m^{+}(k,\ell)$. We need to verify the condition \eqref{bposcond} for $q_0$.
We consider the construction of $q_0$ in  part \ref{cv1}, which is the construction of the so-denoted function $q$ in that part.
Without loss of the generality, we assume $p_0(z,\zeta)$ satisfies  \eqref{simp}.

Thanks to the second property in \eqref{bneg} for $p_0$,  we have  in formula \eqref{simp} for $p_0$ when $\kappa_j=0$ that 
$$\sigma_j(\omega_jz_0)=\cos(0 z_0)\text{ or }\sigma_j(\omega_jz_0)=\sin(0 z_0)
\text{  in \eqref{simp} for $p_0$. }$$
This implies that when computing $p_0 (\widehat{\LL}_k(t),\zeta)$, we have $\omega=0$ for $j=0$ in \eqref{sincos1} and \eqref{sincos2}.
As a consequence,   the imaginary power of $z_{-1}$ in $q_0(z,\zeta)$, which comes from the power $i\omega$ or $-i\omega$ when $j=0$ in \eqref{sincos1}, is in fact zero. This means $\Im(\beta_{-1})=0$ for the resulted formula of $q_0$ as a sum in \eqref{spec2}.
Since $q_0\in \widehat{\mathscr P}_{m}(k,\ell,0,\C,\R)$, we already have $\Re(\beta_{-1})=0$, thus, $\beta_{-1}=0$, and $q_0$ meets the second condition in \eqref{bposcond}. 

Observe that the functions $g$ and $h$ only have imaginary powers for the variable $z$, and the functions $\tilde g$, $\tilde h$ are independent of $z$. Hence,  the vector $\Re\beta$ for $q_0$ in \eqref{spec2} is, in fact, the vector $\beta$ for $p_0$ in \eqref{simp}. Therefore, $q_0$ satisfies the first condition in \eqref{bposcond} thanks to the first property in \eqref{bneg} for $p_0$.  Combining all of these facts, we have $q_0$ satisfies \eqref{bposcond} and, hence,
$q\in \widehat{\mathscr P}_m^{+}(k,\ell)$.
\end{proof}

For the sake of convenience, we say the function $q$ in  \eqref{spec2} satisfies IP($k$). 
More precisely, if $p\in \widehat{\mathscr P}(k,\ell,X)$ for some linear space $X$ over $\C$, then we say $p$ satisfies IP($k$) if it has the form in \eqref{pzedef} with $\Im(\beta_k)=0$ for all power vector $\beta$ in \eqref{pzedef}. Note that $z_k$ must be the last varible in $z$ for $p(z,\zeta)$.
Let $p\in \widehat{\mathscr P}(k,\ell,X)$ satisfy IP($k$) amd $k'\ge k$, $\ell'\ge \ell$.
Then, with the embedding \eqref{PPemd}, $p$ also satisfies IP($k'$) as a function in $\widehat{\mathscr P}(k',\ell',X)$.

\begin{proof}[Proof of Theorem \ref{thm3}]
Because $Z_k\in \widehat{\mathcal P}_{m_*}^+(s_k,k-1) $, we have, thanks to Lemma \ref{convert},\ref{cv3},
\beq\label{ZtoZ}
Z_k (\widehat{\LL}_{s_k}(t),\zeta) 
= \mathscr Z_k (\widehat{\LL}_{s_k}(t),\zeta)  \text{ for some } \mathscr Z_k\in \widehat{\mathscr P}_{m_*}^+(s_k,k-1) 
\text{ that satisfies IP($s_k$).}
\eeq
Thus, $(\mathcal K,(s_k)_{k\in\mathcal K},(\mathscr Z_k)_{k\in\mathcal K})\in \mathscr U(m)$ as defined in Definition \ref{Zsys}.
Let $\mathscr Y_k$ and $\widehat{\mathscr Y}_k$ be defined by  \eqref{Y1}--\eqref{hatY} corresponding to $Z_k:=\mathscr Z_k$.
Then, thanks to \eqref{ZtoZ}, one can verify recursively that
\beq\label{YY}
\mathscr Y_k(t)=\mathcal Y_k(t) \text{ and }\widehat{\mathscr Y}_k(t)=\widehat{\mathcal Y}_k(t) \text{ for all }k\in\mathcal K.
\eeq

For each $n\in\N$, thanks to \eqref{pnreal}, \eqref{spec1} and \eqref{YY} we have 
\beqs
\widehat p_n(\widehat \LL_{M_n}(t),\widehat{\mathcal Y}_{\widetilde M_n}(t))=\widetilde p_n(\widehat \LL_{M_n}(t),\widehat{\mathcal Y}_{\widetilde M_n}(t))
=\widetilde p_n(\widehat \LL_{M_n}(t),\widehat{\mathscr Y}_{\widetilde M_n}(t)) 
\eeqs 
 for some  function
$\widetilde  p_n\in \widehat{\mathscr P}_{m_*}(M_n,\widetilde M_n,0,G_{\alpha,\sigma,\C},G_{\alpha,\sigma})$ that satisfies IP($M_n)$.
Then the asymptotic expansion \eqref{freal} in the sense of Definition \ref{realex} with the functions $(Z_k)_{k\in\mathcal K}$ is actually the asymptotic expansion 
\beqs
f(t)\sim \sum_{n=1}^\infty \widetilde p_n(\widehat{\LL}_{M_n}(t),\widehat{\mathscr Y}_{\widetilde M_n}(t)) L_{m_*}(t)^{-\mu_n} 
\eeqs
in the sense of Definition \ref{Hexpand} with the functions  $(\mathscr Z_k)_{k\in\mathcal K}$.

Let $u(t)$ be any Leray--Hopf weak solution of \eqref{fctnse}.
Applying Theorem \ref{mainthm}, we obtain the asymptotic expansion, in the sense of Definition \ref{Hexpand},
\beq\label{ureal3}
u(t)\sim \sum_{k=1}^\infty \widetilde q_n(\widehat{\LL}_{M_n}(t),\widehat{\mathscr Y}_{\widetilde M_n}(t)) L_{m_*}(t)^{-\mu_n} \text{ in $G_{\alpha+1-\rho,\sigma}$ for any $\rho\in(0,1)$},
\eeq
where $\widetilde q_n\in  \widehat{\mathscr P}_{m_*}(M_n,\widetilde M_n,0,G_{\alpha+1,\sigma,\C},G_{\alpha+1,\sigma})$  for all $n\in\N$.
Thanks to property \eqref{spec3} and \eqref{YY}, we have 
\beq \label{qq1} \widetilde q_n(\widehat{\LL}_{M_n}(t),\widehat{\mathscr Y}_{\widetilde M_n}(t)) 
=\widehat q_n (\widehat{\LL}_{M_n+1}(t),\widehat{\mathscr Y}_{\widetilde M_n}(t))
=\widehat q_n (\widehat{\LL}_{M_n+1}(t),\widehat{\mathcal Y}_{\widetilde M_n}(t))
\eeq 
 for some $\widehat q_n\in \widehat{\mathcal P}_{m_*}(M_n+1,\widetilde M_n,G_{\alpha+1,\sigma})$.
We claim that
\beq  \label{qnIP}
 \widetilde q_n \text{ satisfies IP($M_n$) for all $n\in\N$.}
 \eeq 
 We accept this claim momentarily. By the virtue of \eqref{spec4}, we have $\widehat q_n\in \mathcal P_{m_*}(M_n,\widetilde M_n,G_{\alpha+1,\sigma})$.
Therefore, \eqref{qq1} becomes 
$$ \widetilde q_k(\widehat{\LL}_{M_n}(t),\widehat{\mathscr Y}_{\widetilde M_n}(t)) =\widehat q_n (\widehat{\LL}_{M_n}(t),\widehat{\mathcal Y}_{\widetilde M_n}(t))\text{ with }\widehat q_n\in \widehat{\mathcal P}_{m_*}(M_n,\widetilde M_n,0,G_{\alpha+1,\sigma}).$$
Substituting this equation into the asymptotic expansion \eqref{ureal3}, we obtain the asymptotic expansion \eqref{ureal} in the sense of Definition \ref{realex}.

 We prove the claim \eqref{qnIP} now.  Let $W_k$, for $k\in\mathcal K$, be defined by \eqref{Wk} for $(Z_k)_{k\in\mathcal K}$ being replaced with $(\mathscr Z_k)_{k\in\mathcal K}$.
 Recall that the functions $\widetilde  q_n$ are constructed in Definition \ref{construct} corresponding to these functions $W_k$ and $\widetilde  p_n$.

 Since $\mathscr Z_k$ satisfies IP($s_k)$, so do $\mathcal R\mathscr Z_k$ and $\partial \mathscr Z_k/\partial \zeta_j$ for $1\le j\le k-1$. In particular, $W_1$ satisfies IP($s_1$). Assume $k\ge 1$, and, for all $1\le j\le k-1$, $W_j$ satisfies IP($s_j$).
For $1\le j\le k-1$, because $s_j\le s_k$, we also have each such $W_j$ satisfies IP($s_k$). This yields the product $W_j \partial \mathscr Z_k/\partial \zeta_j$ satisfies IP($s_k$), and, hence,
 $W_k=\mathcal  R\mathscr Z_k+\sum_{j=1}^{k-1}W_j \partial \mathscr Z_k/\partial \zeta_j$ satisfies IP($s_k$).
 By the Induction Principle, one has
 \beq\label{WIP}
 W_k \text{ satisfies IP($s_k$) for all $k\in\mathcal K$.}
 \eeq
 
 Since $\widetilde  p_1$ that satisfies IP($M_1)$, the function  $\widetilde  q_1=\mathcal Z_{A_\C}p_1$ also satisfies IP($M_1$). Assume $n\ge 1$ and $\widetilde  p_k$  satisfies IP($M_k$) for all $1\le k\le n-1$.
For  $1\le k\le n-1$, one has $M_n\ge M_k$, hence, 
\beq \label{qkIP}
\widetilde  q_k\text{   satisfies IP($M_n$). }
\eeq  
Consequently,
\beq\label{BIP}
B(\widetilde  q_m,\widetilde  q_j) \text{ satisfies IP($M_n$) for all $1\le m,j\le n-1$.}
\eeq
 
If $\chi_n=0$ then it satisfies  IP($M_n$). Consider $m_*=0$ and $\chi_n$ is given by the first formula in \eqref{chin} with $\widetilde  q_\lambda$ in the place of $q_\lambda$.
Let $\lambda$ be as in \eqref{chin} and $1\le j\le \widetilde M_\lambda$. 
By \eqref{qkIP} with $k=\lambda$, we have $\mathcal R\widetilde  q_\lambda$ and $\partial \widetilde  q_\lambda/\partial \zeta_j$  
satisfy  IP($M_n$). 
Also, $s_j\le s_{\widetilde M_\lambda}\le s_{\widetilde M_n}\le M_n$, then it follows \eqref{WIP} that
$W_j$ satisfies IP($M_n$). 
Therefore, 
$\chi_n=
\mathcal R \widetilde q_\lambda+\sum_{j=1}^{\widetilde M_\lambda} W_j\frac{\partial \widetilde q_\lambda}{\partial \zeta_j} $ 
satisfies IP($M_n$).

Combining the fact $\widetilde  p_n$, $\chi_n$ satisfing IP($M_n)$ with property \eqref{BIP}, we have
$$
Q_n\eqdef \widetilde p_n - \sum_{\substack{1\le m,j\le n-1,\\ \mu_m+\mu_j=\mu_n}}B_\C(\widetilde q_m,\widetilde q_j) - \chi_n \text{   satisfies IP($M_n$). }
$$
Therefore, $\widetilde q_n=\mathcal Z_{A_\C}Q_n$ satisfies IP($M_n$). 
By the Induction Principle, we obtain \eqref{qnIP}.
\end{proof}

\medskip
\noindent\textbf{Statements and Declarations.} 
The author received no funds or  grants during the preparation of this manuscript. 
This paper contains no data and there are no conflicts of interests.

\bibliography{paperbaseall}{}
\bibliographystyle{abbrv}

\end{document}